\documentclass[a4paper,english]{jnsao}
\usepackage[utf8]{inputenc}
\usepackage{csquotes}
\usepackage{babel}
\usepackage{tikz}
\usepackage{enumitem}
\usepackage{comment}
\usepackage{mathtools}
\usepackage{mdframed}
\usepackage{algpseudocode,algorithm,algorithmicx}
\usepackage{nicecleveref}
\usepackage[textsize=footnotesize,color=DarkOrange!40]{todonotes}
\usepackage{float}
\usepackage{grffile}
\usepackage[section]{placeins}
\manuscriptcopyright{}
\manuscriptlicense{}
\algnewcommand\algorithmicforalll{\textbf{for all}}
\algdef{S}[FOR]{ForAll}[1]{\algorithmicforalll\ #1\ \algorithmicdo}

\newcommand{\dotp}[2]{\langle #1, #2 \rangle}
\mdfdefinestyle{examplebg}{
    backgroundcolor=black!7,%
    hidealllines=true,%
    innertopmargin=-.3em,%
    innerbottommargin=.7em,%
    innerleftmargin=.7em,%
    innerrightmargin=.7em,%
}

\crefname{condition}{Condition}{Conditions}
\theoremstyle{definition}
\newtheorem{assumption}[definition]{Assumption}

\surroundwithmdframed[style=examplebg]{example}
\makeatletter
\crefname{ALC@unique}{line}{lines}
\makeatother

\let\citep\cite

\newcommand{\term}{\emph}

\newcommand{\field}[1]{\mathbb{#1}}
\newcommand{\N}{\mathbb{N}}

\newcommand{\R}{\field{R}}
\newcommand{\extR}{\overline \R}

\newcommand{\norm}[1]{\|#1\|}

\newcommand{\abs}[1]{|#1|}

\newcommand{\inv}[1]{#1^{-1}}

\newcommand{\Union}\bigcup
\newcommand{\Isect}\bigcap
\newcommand{\union}\cup
\newcommand{\isect}\cap
\newcommand{\bigunion}\bigcup
\newcommand{\bigisect}\bigcap

\newcommand{\defeq}{:=}

\newcommand{\downto}{\searrow}
\newcommand{\upto}{\nearrow}

\newcommand{\M}{{M}}
\newcommand{\Imeas}{\mathscr{I}}

\DeclareMathOperator{\Dom}{dom}

\DeclareMathOperator{\TV}{TV}
\DeclareMathOperator{\GSBV}{GSBV}
\DeclareMathOperator{\SBV}{SBV}
\DeclareMathOperator{\BV}{BV}
\DeclareMathOperator{\Lmeasof}{\mathscr{L}}

\DeclareMathOperator{\A}{\mathcal{A}}
\DeclareMathOperator{\Hs}{\mathcal{H}}

\makeatletter
\def \uminus@sym{\setbox0=\hbox{$\cup$}\rlap{\hbox 
        to\wd0{\hss\raise0.5ex\hbox{$\scriptscriptstyle{-}$}\hss}}\box0}
    \def \uminus    {\mathrel{\uminus@sym}}
\makeatother

\makeatletter
\def \weaktostar@sym{\setbox0=\hbox{$\rightharpoonup$}\rlap{\hbox 
        to\wd0{\hss\raise1ex\hbox{$\scriptscriptstyle{*\,}$}\hss}}\box0}
    \def \weaktostar    {\mathrel{\weaktostar@sym}}
\makeatother

\def\BoundMeas{S}
\def\BoundMeasOf{\mathscr{H}^{N-1}}
\def\tBoundMeas{\tilde{S}}
\def\Lmeas{x}
\def\tLmeas{\tilde{x}}

\def\WM{W}

\def\extR{\overline \R}

\def\ee{\mathscr{e}_k}

\DeclareFontFamily{U}{mathx}{\hyphenchar\font45}
\DeclareFontShape{U}{mathx}{m}{n}{<-> mathx10}{}
\DeclareSymbolFont{mathx}{U}{mathx}{m}{n}
\DeclareMathOperator{\dist}{dist}
\DeclareMathOperator{\prox}{prox}

\def\ddd{\,d}

\renewcommand{\prox}[3]{\mathrm{prox}_{#1#2} \left(#3 \right) }

\def\XXint#1#2#3{{\setbox0=\hbox{$#1{#2#3}{\int}$ }
\vcenter{\hbox{$#2#3$ }}\kern-.6\wd0}}

\renewrobustcmd{\downto}{{{\mathchoice%
            {\rotatebox[origin=c]{-20}{$\to$}}% display
            {\rotatebox[origin=c]{-20}{$\to$}}% text
            {\rotatebox[origin=c]{-20}{\scalebox{0.75}{$\to$}}}% subscript
            {\rotatebox[origin=c]{-20}{\scalebox{0.6}{$\to$}}}% subsubscript
}}}

\renewrobustcmd{\upto}{{{\mathchoice%
            {\rotatebox[origin=c]{20}{$\to$}}% display
            {\rotatebox[origin=c]{20}{$\to$}}% text
            {\rotatebox[origin=c]{20}{\scalebox{0.75}{$\to$}}}% subscript
            {\rotatebox[origin=c]{20}{\scalebox{0.6}{$\to$}}}% subsubscript
}}}

\newcommand{\subfloatrecoii}[4]{\subfloat[#1]{\hspace*{#2}\includegraphics[width=#3, trim=70 70 70 70,clip]{#4}\hspace*{#2}}}%
\newcommand{\subfloatrecoi}[3]{\subfloatrecoii{#1}{0.18cm}{#2}{#3}}%
\newcommand{\subfloatrecoiib}[4]{\subfloat[#1]{\hspace*{#2}\includegraphics[width=#3, trim=60 60 60 60,clip]{#4}\hspace*{#2}}}%
\newcommand{\subfloatrecoib}[3]{\subfloatrecoiib{#1}{0.18cm}{#2}{#3}}%
\newcommand{\subfloatcolorbar}[4]{\hspace{#1}\vtop{\vskip#2\hbox{\includegraphics[width=#3,keepaspectratio,trim={0px 0px 0px 0px},clip]{#4}}}}%

\newcommand*{\textlabel}[2]{%
  \edef\@currentlabel{#1}% Set target label
  \phantomsection% Correct hyper reference link
  #1\label{#2}% Print and store label
}
\newcommand{\subcstageb}[9]{%
    \subfloat[#1\vspace{-0.2cm}]%
    {%
        \label{#2}
        \begin{minipage}{#3}%
            \includegraphics[width=1\linewidth, trim=#6 #7 #8 #9,clip]{#4}\\%
            \includegraphics[width=1\linewidth, trim=#6 #7 #8 #9,clip]{#5}%
        \end{minipage}%
    }%
}%
\newcommand{\subcstagee}[6]{%
    {%
        \begin{minipage}{0.18\linewidth}%
            \subfloat[#1\vspace{-0.2cm}]{%
            \includegraphics[width=1\linewidth, trim=130 130 130 130,clip]{#3}}\\%
            \includegraphics[width=1\linewidth, trim=130 130 130 130,clip]{#4}\\%
            \subfloat[#2\vspace{-0.2cm}]{
            \includegraphics[width=1\linewidth, trim=80 80 80 80,clip]{#5}}\\%
            \includegraphics[width=1\linewidth, trim=80 80 80 80,clip]{#6}%
        \end{minipage}%
    }%
}%
\newcommand{\subcstagef}[7]{%
    \subfloat[#1]%
    {%
        \begin{minipage}{0.22\linewidth}%
            \includegraphics[width=1\linewidth, trim=80 80 80 80,clip]{#2}\\%
            \includegraphics[width=1\linewidth, trim=80 80 80 80,clip]{#3}\\%
            \includegraphics[width=1\linewidth, trim=80 80 80 80,clip]{#4}\\%
            \includegraphics[width=1\linewidth, trim=80 80 80 80,clip]{#5}\\%
            \includegraphics[width=1\linewidth, trim=80 80 80 80,clip]{#6}\\%
            \includegraphics[width=1\linewidth, trim=80 80 80 80,clip]{#7}%
        \end{minipage}%
    }%
}%

\def\tabitem{~~\llap{\textbullet}~~}
\def\AreaMSCon{$7.58$}
\def\AreaMSRes{$28.15$}
\def\AreaTVCon{$21.28$}
\def\AreaTVRes{$31.36$}
\def\AreaTVbCon{$54.35$}
\def\AreaTVbRes{$25.18$ }
\def\AreaRealCon{$7.21$}
\def\AreaRealRes{$29.13$}

\author{
    Jyrki Jauhiainen\thanks{Department of Applied Physics, University of Eastern Finland, Kuopio, Finland. \email{jyrki.jauhiainen@uef.fi}}
    \and
    Tuomo Valkonen\thanks{ModeMat, Escuela Politécnica Nacional, Quito, Ecuador \emph{and} Department of Mathematics and Statistics, University of Helsinki, Finland. \email{tuomo.valkonen@iki.fi}}
    \and
    Aku Seppänen\thanks{Department of Applied Physics, University of Eastern Finland, Kuopio, Finland. \email{aku.seppanen@uef.fi}}
}
\title{Mumford-Shah regularization in electrical impedance tomography with complete electrode model}
\shortauthor{Jauhiainen, Valkonen, and Seppänen}
\shorttitle{Mumford–Shah regularization in electrical impedance tomography with CEM}

\acknowledgements{%
    This work was supported by the Academy of Finland, Centre of Excellence of Inverse Modelling and Imaging, 2018–2025 (project 336795).
    T.~Valkonen was supported by the Escuela Politécnica Nacional internal grant PIIF-20-01 as well as the Academy of Finland grant 314701.
}

\begin{document}
%%%%%%%%%%%%%%%%%%%%%%%%%%%%%%%%%%%%%%%%%%%%%%%
\maketitle

\begin{abstract}
    In electrical impedance tomography, we aim to solve the conductivity within a target body
    through electrical measurements made on the surface of the target. This inverse conductivity problem is severely ill-posed, especially in real applications with only partial boundary data available. Thus regularization has to be introduced. Conventionally regularization promoting smooth features is used, however, the Mumford–Shah regularizer familiar for image segmentation is more appropriate for targets consisting of several distinct objects or materials.
    It is, however, numerically challenging.
    We show theoretically through $\Gamma$-convergence that a modification of the Ambrosio–Tortorelli approximation of the Mumford–Shah regularizer is applicable to electrical impedance tomography, in particular the complete electrode model of boundary measurements.
    With numerical and experimental studies, we confirm that this functional works in practice and produces higher quality results than typical regularizations employed in electrical impedance tomography when the conductivity of the target consists of distinct smoothly-varying regions.
\end{abstract}

%%%%%%%%%%%%%%%%%%%%%%%%%%%%%%%%%%%%%%%%%%%%%%%
%%%%%%%%%%%%%%%%%%%%%%%%%%%%%%%%%%%%%%%%%%%%%%%
\section{Introduction}
\label{sec:intro}

Electrical impedance tomography (EIT) is an imaging modality where the electrical conductivity of a target body is inferred from electrical boundary measurements. This problem is often called the \emph{inverse conductivity problem} or \emph{Calderon's problem} \cite{calderon1980inverse,calderon2006inverse}. In abstract terms, Calderon's problem is to determine the conductivity of the target from a Dirichlet to Neumann map $\Lambda_\gamma: H^{1/2}(\partial \Omega) \to H^{-1/2}(\partial \Omega)$, $\Lambda_\gamma f = \gamma \tfrac{\partial u }{\partial \nu}\rvert_{\partial \Omega}$. Parametrized by the conductivity $\gamma$ within the domain $\Omega \subset \R^N$, the latter maps the electrical potentials at the boundary $\partial \Omega$ to electrical currents through the boundary. Inside the domain $ \Omega$, the electric potentials $u$ and the conductivity $\gamma$ are governed by the elliptic partial differential equation (PDE)
\begin{equation}
    \label{eq:calderon}
    \begin{cases}
        \nabla \cdot \gamma \nabla  u = 0 & 
        x \in \Omega\\
        \quad \quad\;\; u = f& 
        x \in \partial \Omega.
    \end{cases}
\end{equation}

\begin{figure}[!tbp]
    \centering
    \includegraphics[width=1.0\linewidth]{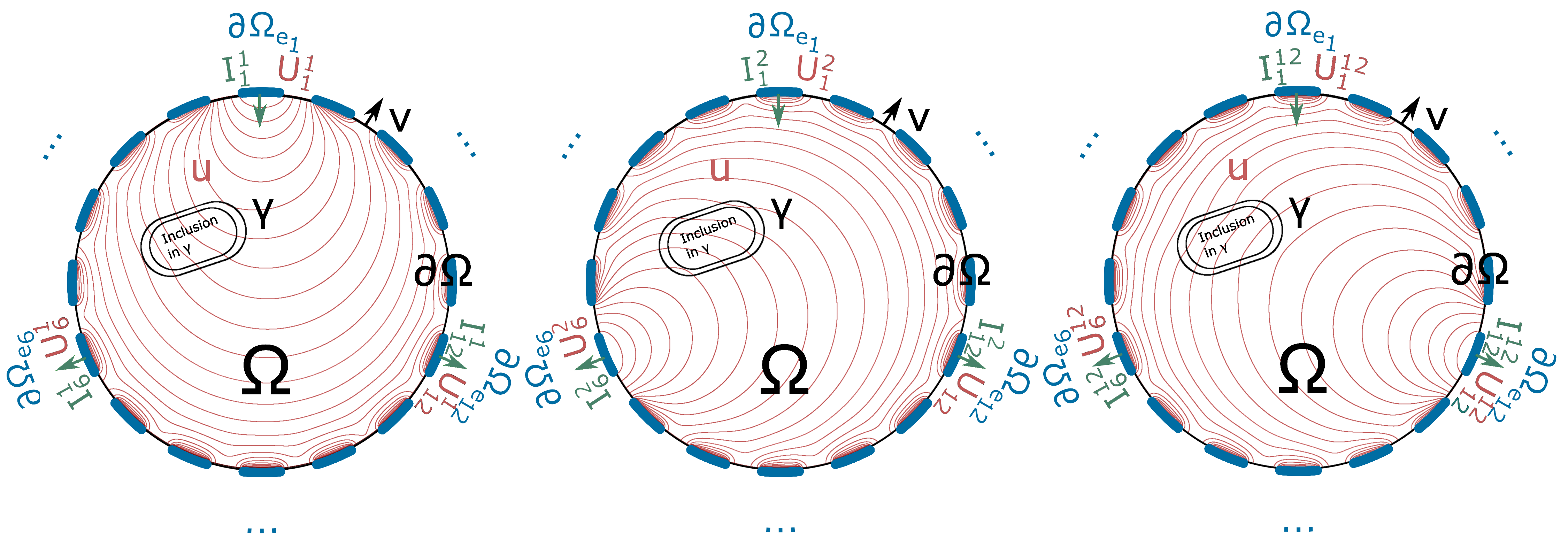}
    \caption{Illustration of the EIT measurement setup. The electrodes are sequentially set to potentials $U^j_k$ and the currents $I^j_k$ caused by the potential $u$ inside the domain $\Omega$ are measured. From the measured currents $I^j_k$, the conductivity $\gamma$ inside the domain is reconstructed. The vector $\nu$ is the outward normal of the boundary $\partial \Omega$. }
    \label{fig:EIT_illu}
\end{figure}

In practice, EIT measurements are collected using finite-sized electrodes, modelled by $\partial \Omega_{e_k}$, on the surface of the target body, performing current injections and measuring the potential differences due to the current injection. This is the more common approach. An alternate approach, which we employ in this paper and that has previously been used in, e.g., \cite{voss2018imaging,voss2019three,jauhiainen2020non,ripgn}, is to measure the currents  $I^j=(I^j_1,I^j_2,\dots,I^j_{P_1})$ caused by the potential excitations $U^j = (U^j_1,U^j_2\dots,U_{P_1}^j)$; see \Cref{fig:EIT_illu}. Usually, multiple sets of potential excitations $U^j$ and current measurements $I^j$, $j=1,\dots, P_2$, are carried out.
To date, the most physically relevant model, adapting \eqref{eq:calderon} to realistic boundary measurements, is the so-called Complete Electrode Model (CEM) \cite{cheng1989electrode}
\begin{subequations}\label{eq:CEM}
    \begin{alignat}{4}
        \label{eq:CEM1}
            \nabla \cdot (\gamma(x) \nabla u^j(x)) &=0  &&
            \quad\text{for } x\in \Omega, \\
        \label{eq:CEM2}
            u^j(x) + \zeta_{k} \gamma(x)\dotp{\nabla u^j(x)}{\nu} &= U^j_{k}   &&
            \quad\text{for } x\in \partial \Omega_{e_{k}}\text{,}\;\; k=1,\dots,P_1, \\
        \label{eq:CEM3}
            \int_{\partial \Omega_{e_k}} \gamma(x) \dotp{\nabla u^j(x)}{\nu}\: \ddd\BoundMeas &= -I^j_{k} &&
            \quad\text{for } k=1,\dots,P_1,\\
            \label{eq:CEM4} 
            \gamma(x)\dotp{\nabla u^j(x)}{\nu} &=0   &&
            \quad\text{for } x\in \partial \Omega \setminus (\partial \Omega_{e_1} \union \ldots \union \partial \Omega_{e_{P_1}}),
    \end{alignat}
\end{subequations}
where $\zeta_{k}$ is the contact impedance of an electrode $k$, i.e. the impedance caused by the interface between the electrode and the medium of the target, and $\nu$ is a unit normal pointing out of $\Omega$. Typically, a weak formulation is employed.

Due to the ill-posedness of the inverse conductivity problem, regularization methods \cite{engl2000regularization} need to be employed to obtain a solution with desired characteristics.
While direct approaches such as Dbar \cite{mueller2012linear} exist, we concentrate on variational regularisation, which incolves an explicit regularizer $F$ constructed to promote desired solution features. The corresponding minimization problem needs to be solved with iterative optimization methods.
To form the variational model, let $\Imeas = (\Imeas^1,\dots,\Imeas^{P_2})\in \R^{\M}$ be the measured currents during multiple different potential excitations $U^1,\ldots,U^{P_2} \in \R^{P_1}$.
It may be that $\M < P_1P_2$ if the measurement device does not measure all the currents $\Imeas_k^j$ during an excitation $j$. 
Let $I(\gamma) \in \R^{\M}$ be the corresponding electric currents $I_k^j$ obtained by solving \eqref{eq:CEM} for given conductivity $\gamma$. Our variational problem is to find $\hat \gamma$ solving
\begin{equation}
    \label{eq:minFG1}
    \min_{\gamma_m \le \gamma \le \gamma_M} F(\gamma) + G(\gamma),
\end{equation}
where $\gamma_m$ and $\gamma_M$ are the bounds for positive and finite electrical conductivity: $F$ is the regularization functional, and the data fidelity is
\begin{equation}\label{eq:G1}
    G(\gamma) = \tfrac{1}{2a}\norm{\WM(I(\gamma)-\mathscr I)}_{2}^2.
\end{equation}
We write $a > 0$ for the regularization parameter and include the weight matrix $\WM \in \R^{M\times M}$ to account for the noise of each measured current.
% Conceptually, solving the problem \eqref{eq:minFG1} is a relatively simple procedure in which we choose a suitable regularization functional based on the available prior information, discretize the conductivity, approximate the forward problem with finite element method (FEM), and solve the discretized problem with a suitable optimization algorithm. While solving the conductivity this way is often very robust, depending on the chosen regularization, the mathematical properties of this problem are not very well known. Indeed, even the existence of this, in the general case with arbitrary $F$, problem is not known.\todo{Disconnect here. Is it even necessary to talk about the mathematical properties?}
    
The choice of regularization functional $F$ has a significant impact on the solution of the inverse EIT problem, e.g., the squared norm of the gradient will impose smooth solutions. However, in practice the target often contains several smoothly-varying materials with distinct edges. An example of such an application is the imaging of concrete structures; the metallic reinforcements \cite{karhunen2010a} and cracks \cite{karhunen2010b} are sharp features while the moisture distribution is spatially smooth \cite{smyl2016}. Another example application is in the industrial process tomography; phase boundaries and diffusive processes, respectively, cause sharp and smooth variations in the conductivity.
Total variation (TV) regularization, i.e., $1$-norm of the (distributional) image gradient, allows sharp edges between materials, but generally suffers from the stair-casing effect: it imposes piecewise constant solutions everywhere.
Total generalized variation (TGV) \cite{bredies2009tgv,sampta2011tgv} can be used to avoid the stair-casing effect while maintaining other desirable characteristics of TV \cite{tuomov-jumpset2,l1tgv}. It has been applied to EIT in \cite{shi2019reduction}.
In practice, however, neither TV nor TGV may not sufficiently well recover edges and distinct objects from very incomplete data.

The Mumford–Shah (M-S) regularizer \cite{mumford1989optimal}, familiar from image segmentation, promotes a small number of distinct smoothly-varying objects with sharp edges to a much higher extent than TV or TGV do. It does this by only penalising the length of object edges instead of the height of the edges.
It is defined by
\begin{equation}
    \label{eq:msf}
    F(\gamma) = \norm{\abs{ \nabla \gamma}}_2^2 + \alpha \mathscr{H}^{N-1}(S_\gamma).
    %\quad (y\in \text{GSBV}(\Omega))
\end{equation}
in the space of (generalized) special functions of bounded variation \cite{ambrosio2000functions}.
This space concerns functions that admit an approximate differential $\nabla \gamma$ outside a \term{jump set} or the \term{approximate discontinuity set}  $S_\gamma$.
Controlled by the parameter $\alpha > 0$, in \eqref{eq:msf} we penalize the length ($N=2$) or area ($N=3$) of the jump set using the Hausdorff measure $\mathscr{H}^{N-1}$ of dimension $N-1$.

For more general settings, Ambrosio and Tortorelli \cite{ambrosio1990approximation} showed that as $k \upto \infty$, the functionals
\begin{equation}
    \label{eq:apmsf}
    F_k(\gamma,z) = \int_\Omega \left( (\abs{\nabla \gamma}^2 + \abs{\nabla z}^2)(1-z^2)^{2k} + \frac{1}{4}(\alpha k z)^2\right) \ddd\Lmeas
    \quad (k > 0)
\end{equation}
approximate $F$ defined in \eqref{eq:msf} in the sense of $\Gamma$-convergence with the underlying topology given by convergence in measure.
We refer to \cite{braides2002gamma} for an introduction to $\Gamma$-convergence.
The variable $z$ in \eqref{eq:apmsf} plays the role of a control variable for the gradient of $\gamma$ and the minimization is done with respect to the both variables $\gamma$ and $z$; that is, the approximate problem is
\begin{equation}
    \label{eq:minFGappI}
    \min_{\gamma,z} G(\gamma) + F_k(\gamma,z).
\end{equation}

The question then is, do the solutions $\gamma_k$ of the approximate problems \eqref{eq:minFGappI} converge to a solution $\gamma$ of the original problem \eqref{eq:minFG1}?
For continuous $G$, the $\Gamma$-convergence of $F_k$ readily implies that of $G+F_k$. If $\{G+F_k\}_{k \in \N}$ are “equi-mild coercive”, it is also possible to show that the solutions converge; see \cite{braides2002gamma}. In \cite{ambrosio1990approximation}, Ambrosio and Tortorelli showed directly for the denoising data term $G(\gamma) = \int_\Omega \abs{\gamma - \gamma^M}^p \ddd\Lmeas$ with $\gamma^M \in L^\infty(\Omega)$ that solutions to \eqref{eq:minFGappI} approximate solution to the problem
\begin{equation}\label{eq:minFGI}
    \min_\gamma G(\gamma) + F(\gamma)
\end{equation}
% More precisely
% \begin{equation}\label{eq:impl1}
%     (\bar \gamma_{k}, \bar z_{k})\quad\text{solves \eqref{eq:minFGappI}}\quad\text{and} \quad \bar \gamma_{k} \to \bar \gamma \text{ in measure } \Rightarrow \bar \gamma \text{ solves \eqref{eq:minFGI}}.
% \end{equation}
More general data terms, such as \eqref{eq:G1}, were not treated.
In \cite{ambrosio1992approximation}, Ambrosio and Tortorelli briefly discussed their treatment, and showed the $\Gamma$-convergence of the alternative approximating functionals
\begin{equation}\label{eq:FeA}
    F_\lambda(\gamma,z) 
    = \int_\Omega \left[\lambda \abs{\nabla z}^2 + \alpha ( z^2 + o_\lambda)\abs{\nabla \gamma}^2 + \frac{(z-1)^2}{4\lambda}\right] \ddd\Lmeas,
\end{equation}
where the approximation takes place as $o_\lambda \to 0$, $\lambda \to 0$ and $\gamma\to 0$ in $L^2$ topology. The benefit of this functional is that it simplifies the numerical implementation.
In contrast to the original approximation \eqref{eq:apmsf}, for which \eqref{eq:minFGappI} may have no solutions, the $o_\lambda$ term in \eqref{eq:FeA} together with the $L^2$-coercivity of $G$ guarantees the existence of solutions to the approximating problems. The approach, however, still has one difficulty with application to the EIT problem: $G$ given by \eqref{eq:G1} can only be proven to be continuous if $\gamma_m \le \gamma \le \gamma_M$ for some constants $\gamma_M > \gamma_m > 0$, whereas the $\Gamma$-convergence proofs of \cite{ambrosio1990approximation,ambrosio1992approximation} specifically depend on $\gamma$ being equal to zero on subdomains. We will, therefore, further need to adapt those proofs.

The Mumford–Shah regularizer and Ambrosio–Tortorelli approximation have previously been studied for the \emph{continuum model} of EIT in \cite{rondi2001enhanced}. In this work, a very strict “Q-property” is imposed on the conductivities $\gamma$ and the aforementioned difficulty with the $\Gamma$-convergence proofs regarding the continuity of $G$ when $\gamma=0$ is circumvented by replacing $G(\gamma)$ by $G(\hat \gamma)$ for a suitably “forced” $\hat\gamma$, and introducing an additional $L^2$ penalty.
Moreover, a drawback of the continuum model of EIT compared to CEM is that it models neither the electrodes nor the contact impedances. In \cite{huska2020spatially}, a reconstruction approach was proposed based on CEM and a regularizer with close appearances to the Ambrosio–Tortorelli functional. However, this approach is based on the  linearization of $G$ and the control variable $z$ is to be obtained \emph{a priori}, either from example from a photograph, or from an initial reconstruction. As such, an asymptotic theory, and the theoretical properties of the Mumford–Shah functional are not available to judge theoretical reconstruction qualities. Nevertheless, the numerical studies suggest that the proposed method outperforms TV.

%Both the jump set and the approximate differential are defined analogously to their exact counterparts by replacing the limits with the approximate limits (see \cite[Definition 3.63]{federer2014geometric} and \cite[Definition 3.70]{ambrosio2000functions}).

Optimization problems involving the Mumford-Shah regularizer \eqref{eq:msf} are in general very challenging due to a high level of nonsmoothness and nonconvexity.
Effective algorithms have been developed for the restriction to piecewise constant functions, also known as the Potts model
\cite{geman1984stochastic,storath2014jump-sparse,storath2015joint,tuomov-nlpdhgm-general}.
For separating two objects, the Chan–Vese convex relaxation \cite{chan2001active} can also be efficiently solved.
Moreover, in \cite{hohm2015algorithmic} the Alternating Directions Method of Multipliers (ADMM) is applied to Mumford–Shah regularized problems with non-linear forward operators. Through a regular finite differences discretisation, the generally expensive ADMM subproblems become a series of one-dimensional Mumford–Shah problems that can be solved efficiently. However, they require that the data term satisfies the coercivity assumption $G(\gamma) \to \infty$ as $\gamma \to \infty$. This is not the case for the EIT/CEM data term. Moreover, the finite differences discretisation is severely limiting when the forward operator involves PDEs on non-rectangular domains.

The aim of this paper is to apply the Ambrosio–Tortorelli approximation \eqref{eq:apmsf} of the Mumford–Shah regularization functional to the complete electrode model of EIT, solving for the control variable $z$ simultaneously with the conductivity $\gamma$.
In \cref{ssec:Gconv}, we show the $\Gamma$-convergence and the convergence of solutions for the approximate EIT problems \eqref{eq:minFGappI}. To ensure the continuity of CEM and the existence of solutions to the approximating problems \eqref{eq:minFGappI}, we will modify $F_k$ slightly by imposing constraints $\gamma_m \le \gamma \le \gamma_M$ and $0 \le z \le 1-\epsilon_k$. The constraint $z \le 1-\epsilon_k$ serves a similar purpose as $o_\lambda$ in \eqref{eq:FeA}. Before we embark on proving $\Gamma$-convergence, we first show in \cref{sec:BP} the continuity of $G$ in measure when $\gamma$ is bounded away from zero.
We finish in \cref{sec:numerical} by evaluating the practical performance of the approach numerically and experimentally. \Cref{table:Symbols} shows explains the notation used in the manuscript.

To keep the length of this manuscript manageable, we do not discuss the $\Gamma$-convergence of finite element approximations to the original problem or the function space Ambrosio–Tortorelli approximations.
Discretisation and application of the Mumford–Shah functional to piecewise constant (Potts model) and piecewise affine functions are studied in \cite{gobbino1998finite,chambolle1999finite,chambolle1999discrete,ramlau2010regularization,klann2013regularization,kiefer2019efficient,foare2019semi}. Of these works, \cite{ramlau2010regularization,klann2013regularization} also discuss applications to computerized and single photon emission tomography. Further in the theme of applications of the Mumford–Shah functional to inverse problems, a level-set method is presented in \cite{ramlau2007mumford} for X-ray tomography. Finally, \cite{weinmann2016mumford} study algorithms for discretised Mumford–Shah regularisation with manifold-valued data.

\begin{table}
    \caption{Common symbols used in the manuscript. }
    \label{table:Symbols}
    \centering
    \footnotesize%\small
    \begin{tabular}{ r l r l}
        Symbol & Explanation & Symbol & Explanation \\
        \hline
        \hline
        $\gamma$ & Electrical conductivity & $\Omega$ & Domain representing the monitored object\\
        $u$ & Electric potential inside $\Omega$ & $\zeta_k$ & Contact impedance of the electrode $k$ \\
        $U_k^j$ & Potential of the electrode $k$ at excitation $j$  & $I_k^j$& Current through the electrode $k$ at excitation $j$ \\
        $P_1$ &Number of excitations  & $P_2$& Number of electrodes \\
        $\mathscr{I}$ & Vector of measured currents & $W$& The weight matrix of the measurements \\
        $F$ & M-S functional with $\gamma$ constraints & $G$ & Data fidelity term \\ 
        $\mathscr{H}$ &  Hausdorff measure & $S_\gamma$ & Jump set of $\gamma$ \\ 
        $\mathscr{B}$ & Space of Borel functions & $\alpha$ & Jump set regularization parameter\\ 
        $z$ & Auxiliary jump set control variable & $k,\lambda$ & Jump set control parameter \\ 
        $F_k$, $F_\lambda$ & A-T approximation of M-S functional & $\bar F_k$ & Modified A-T functional\\ %$\epsilon_k$ & An additional constraint for $z$
        $D_{k,N}$ & Space of the functionals $F_k$ & $\bar D_{k,N}$ & Space of the functionals $\bar F_k$\\
        $\gamma_m,\gamma_M$ & The conductivity constraints & $\epsilon_k$ & An additional constraint variable for $z$ \\
        $H$ & Hilbert space & $\Hs$ & Weak solution space of CEM \\ 
        $B$ & Bilinear form associated with CEM & $L$ & Linear form associated with CEM \\ 
        $\norm{\cdot}^2_{\zeta*}$ & Norm associated with $\Hs$ & $\Lmeasof$ & Lebesgue measure \\
        $G_h$ & Data term of the FE approximation of CEM & $a$ & Typical regularization parameter \\
        $N_n$ & Number of nodes in the FE mesh& $N_e$ & Number of nodes in the FE mesh \\
        $\phi$ & Linear FE basis function & $\delta$ & Indicator function \\
        $f^*$ & Convex conjugate of $f$ & $\text{prox}$ & Proximal mapping \\
        $t^i,s^i$ & Step parameters of the NLPDPS & $w,\beta$ & Step parameters of the RIPGN\\ %         & $a$ & The typical regularization parameter \\ 
    \end{tabular} 
\end{table}%\todo[inline]{Better explanation for Cases 3-8 in the Table 1?} 
%%%%%%%%%%%%%%%%%%%%%%%%%%%%%%%%%%%%%%%%%%%%%%%
\section{Basic properties}
\label{sec:BP}
%%%%%%%%%%%%%%%%%%%%%%%%%%%%%%%%%%%%%%%%%%%%%%%

In this section, we study basic properties of the EIT data term $G$ given by \eqref{eq:G1}, as well as the corresponding approximation problems \eqref{eq:minFGappI}. Specifically, we show that $G$ is continuous in the topology of convergence in measure, as long as $0 <\gamma_m \le \gamma_M < \infty$  almost everywhere. This readily follows from the continuity of each of the individual the currents $I^j$, ($j=1,\ldots,P_2$). Without loss of generality, we therefore concentrate on $P_2=1$, and for brevity drop the measurement setup indicator $j$ from the potentials and currents $u^j$, $I^j$, and $U^j$.

\subsection{The complete electrode model}
\label{ssec:CEM}
%%%%%%%%%%%%%%%%%%%%%%%%%%%%%%%%%%%%%%%%%%%%%%%

We work with the weak formulation of the PDE \eqref{eq:CEM}.
Given a domain $\Omega \subset \R^N$, we define the space $\Hs$ of weak solutions, consisting of both the inner potential $u$ and the currents $I$, as
\[
    \Hs(\Omega)\defeq H^1(\Omega) \oplus \R^{P_1}.
\]
We equip this space with the norm
\begin{equation}\label{eq:Hnorm}
    \norm{(v,V)}^2_{\Hs}\defeq  \norm{v }_{H^1}^2 + \norm{V}_{2}^2\quad \left((v,V) \in \Hs\right),   
\end{equation}
where $\norm{\cdot}_{H^1}$ is the natural norm of the Hilbert space $H^1(\Omega)$.
Now, following \cite{VossThesis2020}, see also \cite{jauhiainen2020non}, a weak solution $(u,I)$ to \eqref{eq:CEM} is characterized by
\begin{equation}\label{eq:bilin}
    B_\gamma((u,I),(v,V)) = L(v,V), \quad \left((v,V) \in \Hs\right),
\end{equation}
where
\begin{equation}\label{eq:B}
    \begin{aligned}[t]
        B_\gamma((u,I),(v,V))&
        = \int_{\Omega} \gamma \dotp{\nabla u}{ \nabla v} \ddd\Lmeas + \sum_{k=1}^{P_1} \tfrac{1}{\zeta_k}\int_{\partial \Omega_{e_k}} uv d \BoundMeas
        \\
        \MoveEqLeft[-1]
         - \sum_{k=1}^{P_1} \tfrac{1}{\zeta_k}\int_{\partial \Omega_{e_k}} uV_k d \BoundMeas + \sum_k^{P_1} I_kV_k,
    \end{aligned}
\end{equation}
and
\begin{equation*}
    L(v,V) = \sum_{k=1}^{P_1} \tfrac{1}{\zeta_k} \int_{\partial \Omega_{e_k}} U_k(v-V_k) \ddd\BoundMeas.
\end{equation*}
The well-posedness of this formulation has been shown in \cite{somersalo1992existence}; see also \cite{jauhiainen2020non}.
Here and throughout, we write $dS$ for integration on the boundary of the relevant domain with respect to the $(N-1)$-dimensional Hausdorff measure. Likewise $dx$ denotes integration with respect to the $N$-dimensional Lebesgue measure $\Lmeasof$.

In the next subsection, we will analyze under what conditions the currents $I$ obtained from \eqref{eq:bilin} are continuous in measure. For this, it will be convenient to define some additional norms besides \eqref{eq:Hnorm}. For $\zeta=(\zeta_1,\ldots,\zeta_{P_1}) \in (0,\infty)$, we define
\[
    \norm{(v,V)}^2_{\zeta*}\defeq \norm{\nabla v}_{2} ^2 + \norm{v}^2_{\partial \Omega_e\zeta} + \norm{V}^2_{2},
    \quad \text{with}\quad 
    \norm{v}^2_{\partial \Omega_e\zeta} \defeq \sum_{k=1}^{P_1} \tfrac{1}{\zeta_k}\int_{\partial \Omega_{e_k}} v^2 d \BoundMeas
\]
with
\[
    \norm{\nabla v}_2^2 \defeq \int_{\Omega} \abs{ \nabla v }^2 \ddd\Lmeas \defeq  \int_{\Omega} \dotp{\nabla v}{ \nabla v} \ddd\Lmeas.
\]
Here $\abs{ \cdot }$ is to be understood as the spatial Euclidean norm at a point $x$ in the integration domain.
If $\Omega$ has Lipschitz boundary, then for some $\Lambda,\lambda > 0$, we have (compare \cite{jauhiainen2020non})
\begin{equation}
    \label{eq:normequiv}
    \Lambda\norm{(v,V)}_{\zeta*} \ge  \norm{(v,V)}_{\Hs} \ge \lambda\norm{(v,V)}_{\zeta*}
    \quad \left((v,V) \in \Hs\right),
\end{equation}

We denote solutions to the weak formulation \eqref{eq:bilin} of the EIT model by $w = (u,I)$ and use the inequalities
\begin{equation}\label{eq:normineq}
    \norm{\nabla u}_{2} \le \norm{w}_{\zeta *} \quad\text{and}\quad \norm{I}_{2} \le \norm{w}_{\zeta *}
\end{equation}
that follow from the definition of $\norm{w}^2_{\zeta *}$.

\subsection{Continuity of the conductivity-to-current maps}\label{ssec:cim}

Let $(\Omega,\Sigma,\mu)$ be a measure space, where $\Sigma$ is $\sigma$-algebra on $\Omega$, and $\mu$ a measure on this $\sigma$-algebra. We say that $\gamma_k\to\gamma$ in measure if for every $\epsilon > 0$,
$$
    \lim\limits_{k\to\infty} \mu(\lbrace x \in \Omega \mid \abs{\gamma(x)-\gamma_k(x) } \ge \epsilon \rbrace) = 0.
$$
We generally take $\mu=\Lmeasof$ the Lebesgue measure on $\Omega \subset \R^N$ and $\Sigma$ the Borel-measurable sets without explicitly stating this.
For $\gamma_M>\gamma_m>0$, we write
\[
    [\gamma_m,\gamma_M]
    \defeq \{ \gamma \in L^\infty(\Omega) \mid \gamma_m \le \gamma \le \gamma_M \text{ a.e.}\},
\]
where the “almost everywhere” or “a.e.” is also with respect to the Lebesgue measure on $\Omega$.
Then assuming that $\gamma \in [\gamma_m,\gamma_M]$, we will show that the electrical currents $I(\gamma_k) \in \R^{P_1}$ converge to $I(\gamma)$ if $\gamma_k \to \gamma$ in (Lebesgue) measure.
We initially work assuming the scaling condition $\zeta^{-1}_k\BoundMeasOf(\partial \Omega_{e_k}) \le 1$, but remove it at the end.
We start by showing the coercivity of $B$ with respect to $\norm{\cdot}_{\zeta*}$

\begin{lemma}
    \label{lemma:coer}
    Suppose $0 < \gamma_m \le \gamma \in L^\infty(\Omega)$ and $\zeta^{-1}_k\BoundMeasOf(\partial \Omega_{e_k}) \le 1$ for all $k=1,\dots,P_1$. Then there exists $C_1 > 0$ independent of $\gamma$ such that
    \begin{equation}\label{eq:Nest1}
        B_\gamma((v,V),(v,V)) \ge C_1\norm{(v,V)}^2_{\zeta*}
        \quad\text{for all}\quad (v,V) \in \Hs.
    \end{equation}
\end{lemma}

\begin{proof}
    Young's inequality and $\zeta^{-1}_k \BoundMeasOf(\partial \Omega_{e_k}) \le 1$ give
    \[
        \tfrac{1}{\zeta_k}\int_{\partial \Omega_{e_k}} vV_k \ddd\BoundMeas 
        \le \tfrac{1}{2\zeta_k}\int_{\partial \Omega_{e_k}} \bigl(  v^2  + V_k^2 \bigr)\ddd\BoundMeas 
        \le \tfrac{1}{2}\bigl(  \tfrac{1}{\zeta_k}\int_{\partial \Omega_{e_k}} v^2 \ddd\BoundMeas  + V_k^2 \bigr).
    \]
    Thus, taking $C_1 = \tfrac{1}{2}\min \left\lbrace1,\gamma_m\right\rbrace$, we have
    \begin{equation*}
        \begin{aligned}[t]
            B_\gamma((v,V),(v,V)) &
            = \int_{\Omega} \gamma \dotp{\nabla v}{ \nabla v} \ddd\Lmeas + \sum_{k=1}^{P_1} \tfrac{1}{\zeta_k}\int_{\partial \Omega_{e_k}} v^2 \ddd\BoundMeas- \sum_{k=1}^{P_1} \tfrac{1}{\zeta_k}\int_{\partial \Omega_{e_k}} vV_k \ddd\BoundMeas + \sum_{k=1}^{P_1} V_k^2 
            \\&
            \ge C_1\left( \int_{\Omega} \dotp{\nabla v}{ \nabla v} \ddd\Lmeas +  \sum_{k=1}^{P_1}\tfrac{1}{\zeta_k} \int_{\partial \Omega_{e_k}} v^2 \ddd\BoundMeas + \sum_{k=1}^{P_1} V_k^2 \right) 
            = C_1\norm{(v,V)}_{\zeta*}^2.
            \mbox{\qedhere}
        \end{aligned}
    \end{equation*}
\end{proof}

We can now establish the well-posedness of \eqref{eq:bilin}.

\begin{lemma}
    \label{lemma:well-posed}
    Suppose $\Omega \subset \R^N$ has Lipschitz boundary, $\gamma_m>0$, and $0 < \zeta^{-1}_k\BoundMeasOf(\partial \Omega_{e_k}) \le 1$ for all $k=1,\dots,P_1$.
    Then for any  $\gamma \in L^\infty(\Omega)$ with $\gamma \ge \gamma_m$ (a.e.), \eqref{eq:bilin} has a unique solution $w=(u,I)$.
\end{lemma}

\begin{proof}
    The equivalence \eqref{eq:normequiv} together with \cref{lemma:coer} establish the coercivity of $B$.
    As $B$ and $L$ are clearly continuous by the same \eqref{eq:normequiv}, the claim follows from the Lax–Milgram theorem.
\end{proof}

In the next lemma, we show that $\norm{\nabla u}_2$ for a solution $(u,I)$ of \eqref{eq:bilin} has an upper bound independent of the conductivity $\gamma$.

\begin{lemma}\label{lemma:CEMBound}
    Suppose $0 < \gamma_m \le \gamma \in L^\infty(\Omega)$. a.e., $\Omega \subset \R^N$ is a bounded Lipschitz domain, and $\zeta^{-1}_k \BoundMeasOf(\partial \Omega_{e_k}) \le 1$ for all $k=1,\dots,P_1$. Then there exists a constant $C_2>0$, independent of $\gamma$, such that any solution $(u,I) \in \Hs$ of \eqref{eq:bilin} satisfies
    \begin{equation}\label{eq:Lest}
        \norm{(u,I)}_{\zeta*} \le C_2\norm{U}_{2},
    \end{equation}
    Moreover, both $\norm{\nabla u}_{2},\norm{I}_{2} \le  C_2\norm{U}_{2}$.
\end{lemma}

\begin{proof}
    The assumption $\zeta^{-1}_k \BoundMeasOf(\partial \Omega_{e_k}) \le 1$ for all $k=1,\dots,P_1$, Young's and Hölder's inequalities give
    \begin{equation}\label{eq:uIineq}
        \tfrac{1}{\zeta_k}\left(\int_{\partial \Omega_{e_k}} (u-I_k) \ddd\BoundMeas\right)^2
        \le
        \tfrac{2}{\zeta_k}\left( \int_{\partial \Omega_{e_k}} u^2 \ddd\BoundMeas +  I_k^2\BoundMeasOf(\partial \Omega_{e_k}) \right)
        \le
        2\left( \tfrac{1}{\zeta_k} \int_{\partial \Omega_{e_k}} u^2 \ddd\BoundMeas +  I_k^2 \right).
    \end{equation}
    Now letting $(u,I)$ be a solution to \eqref{eq:bilin}, as exists by \cref{lemma:well-posed}, and using \cref{lemma:coer} and \eqref{eq:uIineq} we obtain
    \[
        \begin{aligned}[t]
            C_1\norm{(u,I)}^2_{\zeta*}
            &
            \le
            B_\gamma((u,I),(u,I))
            =
            L(u,I)
            =
            \sum_{k=1}^{P_1} \tfrac{1}{\zeta_k} \int_{\partial \Omega_{e_k}} U_k(u-I_k) \ddd\BoundMeas
            \\
            &
            \le
            \norm{U}_{2}\sqrt{ \sum_{k=1}^{P_1} \tfrac{1}{\zeta_k} \left(\int_{\partial \Omega_{e_k}} (u-I_k)d \BoundMeas\right)^2 }
            \le
            \norm{U}_{2}\sqrt{\sum_{k=1}^{P_1}  2\left(\tfrac{1}{\zeta_k} \int_{\partial \Omega_{e_k}} u^2 \ddd\BoundMeas +  I_k^2  \right)}.
        \end{aligned}
    \]
    Finally using $0 \le \int_\Omega \nabla u \cdot \nabla u d \Lmeas$ gives
    \begin{multline*}
        \norm{U}_{2}\sqrt{\sum_{k=1}^{P_1}  2\left(\tfrac{1}{\zeta_k} \int_{\partial \Omega_{e_k}} u^2 \ddd\BoundMeas +  I_k^2  \right)}
        \le
        \norm{U}_{2}\sqrt{ 2\left(  \int_\Omega \nabla u \cdot \nabla u d \Lmeas + \sum_{k=1}^{P_1} \left(\tfrac{1}{\zeta_k}  \int_{\partial \Omega_{e_k}} u^2 \ddd\BoundMeas +  I_k^2 \right)\right)}
        \\
        =
        \sqrt{2}\norm{U}_{2}\norm{(u,I)}_{\zeta*}.
    \end{multline*}
    Altogether, therefore we obtain \eqref{eq:Lest} for $C_2 = \frac{\sqrt{2}}{C_1}$. The inequalities for $\norm{\nabla u}_{2}$ clearly follow from \eqref{eq:normineq}.
\end{proof}

Before showing the continuity of the electrical currents $I$ in measure, we establish two more technical auxiliary results.

\begin{lemma}\label{lemma:cinfbound}
    Suppose that $\zeta^{-1}_k \BoundMeasOf(\partial \Omega_{e_k}) \le 1$ for all $k=1,\dots,P_1$ and that $\Omega \subset \R^N$ is a bounded Lipschitz domain. Given solutions $(u_1,I_1),(u_2,I_2) \in \Hs$ of \eqref{eq:bilin} for the respective conductivities $\gamma_1$ and $\gamma_2$, there exists $C_3 > 0$ independent of $(u_1,I_1)$, $(u_2,I_2)$ such that
    \begin{equation}\label{eq:Gest}
        \left| \int_{\Omega} \gamma \dotp{\nabla u_2}{ \nabla (u_2-u_1)} \ddd\Lmeas \right| \le C_3\norm{\gamma}_{\infty}
        \quad\text{for all}\quad \gamma \in L^\infty(\Omega).
    \end{equation}
\end{lemma}

\begin{proof}
    Let $(u_1,I_1),(u_2,U_2) \in \Hs(\Omega)$. For vectors $ \nabla u_1(x), \nabla u_2(x) \in \R^N$, the Cauchy-Schwartz gives $\abs{\dotp{\nabla u_1(x)}{\nabla u_2(x)}} \le {\abs{\nabla u_1(x)}\abs{\nabla u_2(x)}}$, and Young's inequality gives ${\abs{\nabla u_1(x)}\abs{\nabla u_2(x)}} \le \tfrac{1}{2}(\abs{\nabla u_1(x)}^2 + \abs{\nabla u_2(x)}^2) \le \abs{\nabla u_1(x)}^2 + \abs{\nabla u_2(x)}^2$. Thus, the triangle inequality and \cref{lemma:CEMBound} yield for $C_3 \defeq 3 C_2^2\norm{U}_{2}^2$ that
    \begin{equation}\label{eq:essup}
        \begin{aligned}[t]
            \left| \int_{\Omega} \gamma \dotp{\nabla u_2}{ \nabla (u_2-u_1)} \ddd\Lmeas \right|
            &
            \le
            \norm{\gamma}_{\infty} \bigl( \int_{\Omega}  \left|\dotp{\nabla u_2}{ \nabla (u_2-u_1)}\right| \ddd\Lmeas \bigr)
            \\
            &
            \le
            \norm{\gamma}_{\infty} \left( \int_{\Omega}  \abs{\nabla u_2}^2 + \abs{\nabla u_2}^2 + \abs{\nabla u_1}^2 \ddd\Lmeas \right)
            \\
            &
            \le
            \norm{\gamma}_{\infty} 3 C_2^2\norm{U}_{2}^2 = C_3\norm{\gamma}_{\infty}.
            \mbox{\qedhere}
        \end{aligned}
    \end{equation}
\end{proof}

\begin{lemma}\label{lemma:nunuk}
    Suppose $f,g_k:\Omega \to \R^N$ are such that $\abs{f} \in L^2(\Omega)$ and $\lbrace \abs{ g_k}\rbrace_{k\in\N}$ is bounded in $L^2$. Let $\Omega_k \subset \Omega$ is a sequence of measurable sets such that $\lim_{k \to \infty}\Lmeasof(\Omega_k)= 0$. Then $\lim_{k\to\infty}\int_{\Omega_k} \abs{\dotp{ f}{ g_k}} \ddd\Lmeas = 0$.
\end{lemma}
\begin{proof}
    Since  $\lbrace \abs{ g_k}\rbrace_{k\in\N}$ is bounded in $L^2$, there exists $C > 0$ such that $\int_{\Omega}  \abs{ g_k}^2 \ddd\Lmeas \le C$. Now write $\chi_{\Omega_k}$ for the $\lbrace 0,1 \rbrace$-valued characteristic function of $\Omega_k$. Similarly to the proof of \cref{lemma:cinfbound}, the Cauchy-Schwartz inequality gives $\abs{\dotp{ f}{ g_k}} \le \sqrt{\abs{ f}^2\abs{ g_k}^2}$. Further, Hölder's inequality and the fact that $\int_{\Omega_k} \abs{ f}\abs{ g_k} \ddd\Lmeas = \int_{\Omega} \chi_{\Omega_k} \abs{ f}\abs{ g_k} \ddd\Lmeas$ gives
    \[
        \begin{aligned}[t]
            \int_{\Omega_k} \abs{\dotp{ f}{ g_k}} \ddd\Lmeas
            &
            \le \int_{\Omega} \sqrt{(\abs{ f}\chi_{\Omega_k})^2 \abs{ g_k}^2} \ddd\Lmeas
            \\
            &
            \le \sqrt{\int_{\Omega} (\abs{ f }\chi_{\Omega_k})^2 \ddd\Lmeas} \sqrt{\int_{\Omega} \abs{ g_k}^2 \ddd\Lmeas} \le \sqrt{\int_{\Omega_k} \abs{ f}^2 \ddd\Lmeas} \sqrt{C}.
        \end{aligned}
    \]
    It is well-known (see Corollary 16.9 in \cite{jost2006postmodern}) that given a Lebesgue integrable function  $f: \Omega \to [-\infty,\infty]$ and a sequence of measurable sets $\Omega_k \subset \Omega$ such that $\lim_{k \to \infty}\Lmeasof(\Omega_k)= 0$, then $\lim_{k\to \infty} \int_{\Omega_k} f \ddd\Lmeas = 0$.
    Now since $\lim_{k \to \infty}\Lmeasof(\Omega_k)= 0$, the conditions of this result are satisfied for $\abs{f}^2$ and $\Omega_k$, and thus $\int_{\Omega_k} \abs{ f}^2 \ddd\Lmeas\to 0$ as $k\to \infty$, meaning that also $\sqrt{\int_{\Omega_k} \abs{ f}^2 \ddd\Lmeas} \sqrt{C} \to 0$ yielding
    $$
        0\le \lim_{k\to \infty}\int_{\Omega_k} \abs{\dotp{ f}{ g_k}} \ddd\Lmeas \le \lim_{k\to \infty} \sqrt{\int_{\Omega_k} \abs{ f}^2 \ddd\Lmeas} \sqrt{C}=0.
    $$ 
    This finishes the proof.
\end{proof}

We are now ready to show the continuity of $I$ in measure. 

\begin{lemma}\label{lemma:cim0}
    Suppose that $\zeta^{-1}_k \BoundMeasOf(\partial \Omega_{e_k}) \le 1$ for all $k=1,\dots,P_1$ and $\Omega \subset \R^N$ is a bounded Lipschitz domain. Let $\gamma,\gamma_k \in [\gamma_m,\gamma_M]$ be such that $\gamma_k\to\gamma$ in (Lebesgue) measure. Then $I(\gamma_k) \to I(\gamma)$.
\end{lemma}
\begin{proof}    
    Let $w \defeq (u,I)$ and $w_k\defeq (u_k,I_k)$ be the solutions to \eqref{eq:bilin} for the respective conductivities $\gamma$ and $\gamma_k$, as exist by \cref{lemma:well-posed}. Then
    \begin{equation}\label{eq:Best}
        \begin{aligned}[t]
            B_{\gamma_k}(w - w_k, w - w_k) &= B_{\gamma_k}(w, w - w_k) - B_{\gamma_k}(w_k, w - w_k)
            \\
            &
            = B_{\gamma_k}(w, w - w_k) - L(w - w_k)
            \\&= B_{\gamma_k}(w, w - w_k) - B_{\gamma}(w, w - w_k)
            \\
            &
            = \int_{\Omega} (\gamma_k - \gamma) \dotp{\nabla u}{ \nabla (u-u_k)} \ddd\Lmeas.
        \end{aligned}
    \end{equation}
    By \cref{lemma:coer}, \eqref{eq:normineq}, and \eqref{eq:Best},
    \begin{equation}
        \label{eq:cimA}
        \begin{aligned}[t]
        \norm{I(\gamma)-I(\gamma_k)}_{2}^2
        &
        \le \norm{w_k-w}_{\zeta^*}^2
        \\
        &
        \le C_1^{-1}B_{\gamma_k}(w-w_k, w-w_k)
        \le C_1^{-1} \int_{\Omega} \abs{\gamma_k - \gamma} \abs{\dotp{\nabla u}{ \nabla (u-u_k)}} \ddd\Lmeas.
        \end{aligned}
    \end{equation}
    Let
    \[
        \Omega_{k}^< \defeq \lbrace x \in \Omega \mid \abs{\gamma_k -\gamma} < \epsilon_0  \rbrace
        \quad\text{and}\quad
        \Omega_{k}^\ge \defeq \lbrace x \in \Omega \mid \abs{\gamma_k -\gamma} \ge \epsilon_0 \rbrace.
    \]
    Since  $\norm{\gamma_k - \gamma}_\infty < \epsilon_0$ in $\Omega_{k}^\ge$, by \cref{lemma:cinfbound},
    \begin{equation}\label{eq:cimB}%\tag{B}
        \begin{aligned}[t]
            & C_1^{-1} \int_{\Omega_{k}^<} \abs{\gamma_k - \gamma} \abs{\dotp{\nabla u}{ \nabla (u-u_k)}} \ddd\Lmeas
             <   C_1^{-1}  C_3 \epsilon_0.
        \end{aligned}        
    \end{equation}
    Further, in the set $\Omega^\ge_k$, using triangle inequality and $\gamma,\gamma_k \in [\gamma_m,\gamma_M]$ a.e., we have that
    \begin{equation}\label{eq:cimC}
        \begin{aligned}[t]
            & C_1^{-1} \int_{\Omega_{k}^\ge} \abs{\gamma_k - \gamma} \abs{\dotp{\nabla u}{ \nabla (u-u_k)}} \ddd\Lmeas
            \le  C_1^{-1} \abs{\gamma_M -\gamma_m}\int_{\Omega_{k}^\ge}  \abs{\dotp{\nabla u}{ \nabla (u-u_k)}} \ddd\Lmeas.
        \end{aligned}        
    \end{equation}
    Since $\gamma_k \to \gamma$ in measure, $\Lmeasof(\Omega_{k}^\ge) \to  0$ as $k\to\infty$. Let then $f \defeq C_1^{-1} \abs{\gamma_M -\gamma_m} {\nabla u}$ and $g_k \defeq {\nabla (u-u_k)}$. Clearly $f \in L^2(\Omega)$ (since $u \in H^1(\Omega)$), moreover by triangle inequality, Hölder's inequality, and \cref{lemma:CEMBound}, we have that 
    $$
        \abs{g_k} \le 2 C_1^{-1} \abs{\gamma_M -\gamma_m}\Lmeasof(\Omega) \sqrt{C_2\norm{U}_{2}},
    $$
    Clearly the conditions of \cref{lemma:nunuk} are satisfied for $\Lmeasof(\Omega_{k}^\ge)$, $f$, and $\lbrace g_k\rbrace_{k\in\N}$ meaning that for every $\epsilon_1 > 0$ there exists a $k_1$ so that $ C_1^{-1} \abs{\gamma_M -\gamma_m}\int_{\Omega_{k}^\ge}  \abs{\dotp{\nabla u}{ \nabla (u-u_k)}} \ddd\Lmeas < \epsilon_1$ when $k \ge k_1$. Let $\epsilon > 0$. Take $\epsilon_0 = C_1\epsilon^2/(2C_3)$, $\epsilon_1 = \epsilon^2/2$ and let $k$ to be large enough so that $k \ge k_1$. Combining \cref{eq:cimA,eq:cimB,eq:cimC}
    \[
        \begin{aligned}[t]
            \norm{I(\gamma)-I(\gamma_k)}_{2}^2
            &
            \le \norm{w_k-w}_{\zeta^*}^2
            \\
            &
            \le C_1^{-1}\left| \int_{M} (\gamma_k - \gamma) \dotp{\nabla u}{ \nabla (u-u_k)} \ddd\Lmeas \right|
            <   C_1^{-1}  C_3 \epsilon_0 + \epsilon_1 \le \epsilon^2.
        \end{aligned}
    \]
    Since $\epsilon > 0$ was arbitrary, we deduce, as claimed, that $I(\gamma_k) \to I(\gamma)$.
\end{proof}

It is easy to see that \Cref{lemma:coer} does not hold in general; for example, take $\zeta^{-1}_k = 4, \BoundMeasOf(\partial \Omega_{e_k}), v=1/2,$ and $ V_k = 1$. Next, however, we will show that whenever we are concerned with solutions to \eqref{eq:bilin}, the condition $\zeta^{-1}_k\BoundMeasOf(\partial \Omega_{e_k}) \le 1$ is without loss of generality, as every solution in $\Hs( \Omega)$ corresponds to a solution in $\Hs(\tilde \Omega)$ on a domain $\tilde\Omega$ for which this condition holds. Next we use this argument to generalize the results of \cref{lemma:cim0}. We note that the same argument can be also used to generalize \cref{lemma:CEMBound,lemma:cinfbound} but doing so is not necessary for the continuity proof.

\begin{theorem}\label{thm:cim}
    Suppose that $\Omega \subset \R^N$ is a bounded Lipschitz domain. Let $\gamma,\gamma_k \in [\gamma_m,\gamma_M]$ be such that $\gamma_k\to\gamma$ in (Lebesgue) measure.
    Then $I(\gamma_k) \to I(\gamma)$.
\end{theorem}
\begin{proof}
    Take $\tilde \Omega = \lbrace c x \mid x \in \Omega \rbrace$ for
    $
        c \defeq (\check\zeta/\check S)^{1/(N-1)}
    $
    with
    $
        \check\zeta \defeq \min_{k} \zeta_k
    $
    and
    $
        \check S \defeq \max_{k}\BoundMeasOf(\partial \Omega_{e_k}).
    $

    Functions in $u \in H^1( \Omega)$ are in one-to-one correspondence with functions $\tilde u \in H^1(\tilde \Omega)$ through the transformation $\tilde u= u( \inv c x) $, i.e., $\tilde u(cx) = u(x)$.
    So suppose $( u, I)$ solves \eqref{eq:bilin} on $\Hs( \Omega)$ for $\gamma$ and let $\tilde \gamma(x) = c \gamma(\inv c x)$. 
    We will show that $(\tilde u, \tilde I)\defeq (\tilde u, c^{N-1}I) \in \Hs(\Omega)$ solves \eqref{eq:bilin} on $\Hs( \tilde\Omega)$ for $\tilde \gamma$.
    Let $(\tilde v, V) \in \Hs(\tilde \Omega)$ be an arbitrary test function on $\Hs(\tilde \Omega)$ and define $v(x) = \tilde v(c^{-1}x)$ so that $(v, V) \in \Hs(\Omega)$.
    A change of variables shows that
    \[
        \begin{aligned}[t]
        B_{\tilde \gamma}((\tilde u, \tilde I), (\tilde v, V))
        &
        =
        \int_{\tilde \Omega} \tilde \gamma \dotp{\nabla \tilde u}{\nabla \tilde v} d\tLmeas + \sum_{k=1}^{P_1} \tfrac{1}{\zeta_k}\int_{\partial \tilde \Omega_{e_k}} \tilde u \tilde v \ddd\BoundMeas
        - \sum_{k=1}^{P_1} \tfrac{1}{\zeta_k}\int_{\partial \tilde \Omega_{e_k}} \tilde u V_k \ddd\tBoundMeas + \sum_k^{P_1} \tilde I_kV_k
        \\
        &
        =
        \int_{\Omega} c\gamma \dotp{c^{-1}\nabla u}{c^{-1}\nabla v} c^{N}\ddd\Lmeas + \sum_{k=1}^{P_1} \tfrac{c^{N-1}}{\zeta_k}\int_{\partial \Omega_{e_k}} uv \ddd\BoundMeas
        \\
        \MoveEqLeft[-1]
        - \sum_{k=1}^{P_1} \tfrac{c^{N-1}}{\zeta_k}\int_{\partial \Omega_{e_k}} uV_k \ddd\BoundMeas
        + \sum_k^{P_1} c^{N-1}I_kV_k
        \\
        &
        = c^{N-1} B_\gamma((u, I), (v, V))
        \end{aligned}
    \]
    whereas
    \[
        \begin{aligned}[t]
        L(\tilde v, V)
        &
        = \sum_{k=1}^{P_1} \tfrac{1}{\zeta_k} \int_{\partial \tilde  \Omega_{e_k}} U_k(\tilde v-V_k) \ddd\BoundMeas
        = \sum_{k=1}^{P_1} \tfrac{c^{N-1}}{\zeta_k} \int_{\partial \Omega_{e_k}} U_k( v-V_k) \ddd\BoundMeas
        = c^{N-1}L(v, V).
        \end{aligned}
    \]
    Since $(u,I)$ solves \eqref{eq:bilin}, $B_\gamma((u,I),(v,V)) = L(v,V)$ for any $(v,V) \in \Hs(\Omega)$. Particularly this holds to our choice $(v,V)$ showing that $B_{\tilde \gamma}((\tilde u,\tilde I),(\tilde v,V)) = L(\tilde v,V)$. Since $(\tilde v, V) \in \Hs(\tilde \Omega)$ was arbitrary, this proves that $(\tilde u, \tilde I)$ solves \eqref{eq:bilin} on $\Hs( \tilde\Omega)$ for $\tilde \gamma$.

    We have
    $$
        \begin{aligned}[t]
        \zeta^{-1}_k\tilde  \xi(\partial \tilde \Omega_{e_k})
        &
        = c^{1-N} \zeta^{-1}_k\BoundMeasOf(\partial \Omega_{e_k})
        \\
        &
        \le c^{1-N} \zeta^{-1}_m \BoundMeasOf(\partial \Omega_{e_M}) = \tfrac{\check\zeta}{\check S} (\min_{k} \zeta_k)^{-1}_m \max_{k}\BoundMeasOf(\partial \Omega_{e_k}) = 1.
        \end{aligned}
    $$
    Therefore $\tilde\Omega$ satisfies the assumption $\zeta^{-1}_k\tilde  \xi(\partial \tilde \Omega_{e_k}) \le 1$ of \cref{lemma:cim0} for all $k$.
    Now, given that $\gamma_k \to \gamma$ in measure with $\gamma_k,\gamma \in [\gamma_m,\gamma_M]$, we have $c\gamma_k,c\gamma \in [c\gamma_m,c\gamma_M]$ and $c\gamma_k \to c\gamma$ in measure. Since $\zeta^{-1}_k\tilde  \xi(\partial \tilde \Omega_{e_k}) \le 1$ holds for all $k$, \cref{lemma:cim0} shows that $c^{N-1} I(\gamma_k) \to c^{N-1}I(\gamma)$, establishing the claim. The well-posedness of \eqref{eq:bilin} on $\Omega$ is established by instead supposing that $(\tilde u,\tilde I)$ solves \eqref{eq:bilin} on $\Hs(\tilde \Omega)$ for $\tilde \gamma$ and letting $(v,V)\in \Hs( \Omega)$.
\end{proof}
\begin{remark}
    It is a consequence of \cref{thm:cim} that if $\gamma,\gamma_k \in [\gamma_m,\gamma_M]$, then $I$ is also continuous w.r.t. $\norm{\cdot}_p$, $1 \le p \le \infty$, since $L^p$ convergence implies convergence in measure (through Tchebychev's inequality).
\end{remark}

%%%%%%%%%%%%%%%%%%%%%%%%%%%%%%%%%%%%%%%%%%%%%%%
\section{Approximation of the Mumford–Shah functional}
\label{ssec:Gconv}
%%%%%%%%%%%%%%%%%%%%%%%%%%%%%%%%%%%%%%%%%%%%%%%

In this section, we analyze the approximation \eqref{eq:apmsf} of the M-S regularizer \eqref{eq:msf} in the EIT reconstruction problem.
Recall that the data term reads
\begin{equation}\label{eq:G}
    G(\gamma) \defeq \tfrac{1}{2a}\norm{\WM( I(\gamma) - \Imeas)}_{2}^2.
\end{equation}
With $F$ given by \eqref{eq:msf}, our goal is to show the convergence of solutions of the approximate problems
\begin{equation}\label{eq:minFGapp}
    \min\limits_{\gamma,z} \bar F_k(\gamma,z) + G(\gamma),\quad(k > 0)
\end{equation}
to solutions of the target problem
\begin{equation}\label{eq:minFG}
    \min\limits_{\gamma} F(\gamma) + G(\gamma).
\end{equation}

%%%%%%%%%%%%%%%%%%%%%%%%%%%%%%%%%%%%%%%%%%%%%%%
\subsection{Setting up the approximation}
%%%%%%%%%%%%%%%%%%%%%%%%%%%%%%%%%%%%%%%%%%%%%%%

We cannot use the results of Ambrosio and Tortorelli from \cite{ambrosio1990approximation} directly, since their work, in parts, depends on the data fidelity term $G$ being a simple power-of-norm distance with no forward operator. Further, \cite{ambrosio1992approximation}, while amenable to the treatment of general smooth fidelities, doesn't significantly help us, as it still cannot directly handle the constraint $\gamma \ge \gamma_m>0$.
We will, however, base our work on \cite{ambrosio1990approximation} as far as possible, and adapt the rest. To do so, we replace the domain
\begin{equation*}
    D_{k,N}(\Omega) \defeq \left\lbrace (\gamma,z) \in \mathscr{B}(\Omega) \times \mathscr{B}(\Omega)\mid \phi \circ z \in H^{1}(\Omega), \Psi(n \wedge \gamma \vee -n, z) \in H^{1}(\Omega)\; \forall n \in \N \right\rbrace,
\end{equation*}
where $\Omega \subset \R^N$, $\phi(t) \defeq \int_0^1 (1-s^2)^k ds$, and $\Psi(s,t)\defeq s(1-t)^{h+1}$, of the approximating functionals $F_k$ of \eqref{eq:apmsf} by
\[
    \bar D_{k,N}(\Omega) \defeq \lbrace (\gamma,z) \in D_{k,N}(\Omega) \mid \gamma \in [\gamma_m,\gamma_M] \rbrace.
\]
We also replace $F_k$ by the slightly modified
\begin{equation}\label{eq:Fk1app}
   \bar F_k(\gamma,z) \defeq  \begin{cases} \int_\Omega  (\abs{\nabla \gamma}^2 + \abs{\nabla z}^2)(1-z^2)^{2k} + \frac{1}{4}(\alpha k z)^2 \ddd\Lmeas, & (\gamma,z)  \in \bar D_{k}(\Omega),z\le 1-\epsilon_k,
   \\
   \infty, & \mbox{otherwise, } \end{cases}
\end{equation}
where $k\in\N$ and $\epsilon_k>0$.
This differes from $F_k$ of  \cite{ambrosio1990approximation} only by the bounds $\gamma \in [\gamma_m,\gamma_M]$ and $z \le 1-\epsilon_k$, and reduces to that when $\gamma \in [\gamma_m,\gamma_m]$ and $\epsilon_k=0$.
These bounds are needed to ensure the existence and boundedness of solutions to \eqref{eq:minFGapp} with $G$ as in \eqref{eq:G}.

Since we need to restrict the conductivities $\gamma \in [\gamma_m,\gamma_M]$, we replace the space $\GSBV(\Omega)$ of generalized functions of bounded variation, used in \cite{ambrosio1990approximation} by $\SBV(\Omega)$, the space of special functions of bounded variation, consisting of all functions of bounded variation $\gamma \in \BV(\Omega)$ with zero Cantor part in the distributional derivative.
Indeed, $\GSBV(\Omega) \cap [\gamma_m,\gamma_M] = \SBV(\Omega) \cap [\gamma_m,\gamma_M]$.
We refer to \cite{ambrosio2000functions} for detailed definitions of these spaces.

To expand $F$ of \eqref{eq:msf} to the space $(\gamma,z) \in \mathscr B(\Omega)  \times \mathscr B(\Omega)$, using (superfluous) control $z$ we define
\begin{equation}\label{eq:defF}
    F(\gamma,z) \defeq  \begin{cases} \int_\Omega \abs{ \nabla \gamma}^2\ddd\Lmeas + \alpha  \mathscr H^{N-1}(S_\gamma), & \gamma  \in SBV(\Omega),\gamma \in [\gamma_m,\gamma_M], z=0 \\ +\infty, & \mbox{otherwise. } \end{cases}
\end{equation}

Before proving the $\Gamma$-convergence of $\bar F_k+G_k$ to $F+G$, we show that the approximating problems have solutions.

\begin{theorem}\label{thm:FGaexists}
    Suppose that $\Omega \subset \R^N$ is a bounded Lipschitz domain. Fix $k \in \N$ and $0 < \epsilon_k < 1$. Let $\bar F_k$ be given by \eqref{eq:Fk1app} and $G$ by \eqref{eq:G}. Then \eqref{eq:minFGapp} has a solution in $\bar D_{k,N}(\Omega)$.
\end{theorem}

\begin{proof}
    Let $J(\gamma, z) \defeq G(\gamma) + \bar F_k(\gamma,z)$ and denote the $G(\gamma_m)$-sublevel set of $J$ by
    \begin{equation*}
        K\defeq \left\lbrace(\gamma,z) \in D_{k,N}(\Omega) \mid J(\gamma,z) \le G(\gamma_m) \right\rbrace.
    \end{equation*}
    Then $K$ is non-empty since $G(\gamma_m) + F(\gamma_m,0) = G(\gamma_m)$, i.e. $(\gamma_m,0) \in K$, and $\inf_{(\gamma,z)\in K} J(\gamma,z)= \inf_{(\gamma,z)\in D_{k,N}(\Omega)} J(\gamma,z)$.

    Let $(\gamma,z) \in K$. Since $G(\gamma) + \bar F_k(\gamma,z) \le G(\gamma_m) < \infty$, $\gamma \in [\gamma_m,\gamma_M]$ and $z \in [0,1-\epsilon_k]$. Further, we have
    \[
        \begin{aligned}[t]
            \int_\Omega (\abs{\nabla \gamma}^2 + \abs{\nabla z}^2)\epsilon_k^{2k} \ddd\Lmeas &= \int_\Omega  (\abs{\nabla \gamma}^2 + \abs{\nabla z}^2)((2-1)\epsilon_k)^{2k} \ddd\Lmeas
            \le \int_\Omega  (\abs{\nabla \gamma}^2 + \abs{\nabla z}^2)((2-\epsilon_k)\epsilon_k)^{2k} \ddd\Lmeas \\ &= \int_\Omega  (\abs{\nabla \gamma}^2 + \abs{\nabla z}^2)(1-(1-\epsilon_k)^2)^{2k} \ddd\Lmeas \le \int_\Omega  (\abs{\nabla \gamma}^2 + \abs{\nabla z}^2)(1-z^2)^{2k} \ddd\Lmeas,
        \end{aligned}
    \]
    Due to the definition of $K$, it follows
    \[
        \int_\Omega (\abs{\nabla \gamma}^2 + \abs{\nabla z}^2) \ddd\Lmeas \le \epsilon_k^{-2k}(F_k(\gamma,z) + G(\gamma))\le G(\gamma_m)\epsilon_k^{-2k},
    \]
    implying that both $\gamma$ and $z$ are bounded in $H^1(\Omega)$. 
    
    Let $\lbrace (\gamma_i, z_i) \rbrace_{i\in\N} \subset K$,  be a minimizing sequence for $J$, satisfying
    \[
        J(\gamma_i, z_i) - 1/i \le \inf_{(\tilde \gamma, \tilde z)\in D_{k,N}(\Omega)} J(\tilde \gamma, \tilde z)\defeq m
        \quad\text{for all } i \in \N.
    \]
    Since we have shown both $\{\gamma_i\}_{i \in \N}$ and $\{z_i\}_{i \in \N}$ to be bounded in $H^1(\Omega)$, $\lbrace (\gamma_i, z_i) \rbrace_{i\in\N}$ has a subsequence, unrelabelled, convergent to some $(\gamma, z) \in H^1(\Omega) \times H^1(\Omega)$ a.e. and consequently also in measure \cite[Theorem 17.15]{jost2006postmodern}. The (almost everywhere) bounds $\gamma _i\in [\gamma_m,\gamma_M]$ and $z_i \in [0,1-\epsilon_k]$ 
    ensure $\gamma \in [\gamma_m,\gamma_M]$ and $z \in [0,1-\epsilon_k]$ (almost everywhere). Further, since $z \in H^1(\Omega)$, $\phi(s) = \int_0^t(1-s^2)^k \in C^1([0,1])$, $\phi' \in L^\infty([0,1])$, and $\phi(0)=0$ by \cite[Lemma 1.25]{KinunenSobolev}
    $f \circ z \in H^1(\Omega)$ so that $(\gamma,z) \in \bar D_{k,N}(\Omega)$.

    Since $\Omega \subset \R^N$ is a bounded Lipschitz domain, $\gamma,\gamma_i \in [\gamma_m,\gamma_M]$, and $\gamma_i\to \gamma$ in measure, by \cref{thm:cim} $G(\gamma_i) \to G(\gamma)$. Further, $\gamma,\gamma_i \in [\gamma_m,\gamma_M]$ and $z,z_i \in [0,1-\epsilon_k]\subset [0,1)$ yield $\bar F_k(\gamma, z) = F_k(\gamma, z)$ and $\bar F_k(\gamma_i, z_i) = F_k(\gamma_i, z_i)$.
    Thus the lower semicontinuity of $F_k$ \cite[Theorem 3.4]{ambrosio1990approximation} establishes the lower semicontinuity of $J$. By standard arguments $J(y)=m$, establishing the claim.

\end{proof}

%%%%%%%%%%%%%%%%%%%%%%%%%%%%%%%%%%%%%%%%%%%%%%%
\subsection{\texorpdfstring{$\Gamma$}{Gamma}-convergence of the regularization functionals}
%%%%%%%%%%%%%%%%%%%%%%%%%%%%%%%%%%%%%%%%%%%%%%%

Next we disuss the $\Gamma$-convergence of $\bar F_k$ to (the extended definition in \eqref{eq:defF} of) $F$ in $\mathscr{B}(\Omega) \times \mathscr{B}(\Omega)$. Together with the continuity of the currents, this property allows us to show the convergence of the solutions of \eqref{eq:minFGapp} to \eqref{eq:minFG}.

$\Gamma$-convergence means that so-called the $\Gamma$-liminf and the $\Gamma$-limsup inequalities hold for $\bar F_k$ and $F$.
The former is to say that for any $\mathscr{B}(\Omega) \times \mathscr{B}(\Omega) \ni (\gamma_k,z_k) \to (\gamma,z) \in \mathscr{B}(\Omega) \times \mathscr{B}(\Omega)$ (in measure) we have
\begin{equation}\label{eq:Gamma1}
    \liminf\limits_{k \to \infty} \bar F_k(\gamma_k,z_k) \ge F(\gamma,z).
\end{equation}
As for $\Gamma$-limsup, we require that for all $(\gamma,0) \subset \mathscr{B}(\Omega) \times \mathscr{B}(\Omega)$, there exists a \emph{reconstruction} sequence $\lbrace (\gamma_k,z_k) \rbrace_{k\in\N} \subset \mathscr{B}(\Omega) \times \mathscr{B}(\Omega)$, such that $(\gamma_k,z_k)\to (\gamma,0)$ in measure and
\begin{equation}\label{eq:Gamma2}
    \limsup\limits_{k \to \infty} \bar F_k(\gamma_k,z_k) \le F(\gamma,0).
\end{equation}
We refer to \cite{braides2002gamma} for a more comprehensive introduction to $\Gamma$-convergence.

The $\Gamma$-convergence of $F_k$ to $F$, shown originally by Ambrosio and Tortorelli in \cite{ambrosio1990approximation}, holds under the following reflection condition:

\begin{assumption}\label{Ocond}
    $\Omega \subset \R^{N}$ is an open bounded Lipschitz domain and, moreover, there exists a neighborhood $U$ of $\partial \Omega$ and an injective bi-Lipschitz $\phi:U \cap \Omega \to U\backslash\bar \Omega$ with
    \[
        \lim_{y\to x} \phi(y) = x
        \quad\text{for all}\quad x\in \partial \Omega.
    \]
\end{assumption}
This condition allows us to extend the domain of $\gamma$ from $\Omega$ to $\Omega \cup U$ by reflecting the values of $\gamma$ in $\Omega \cap U$ to $U\backslash \bar \Omega$. It is satisfied e.g. by a ball in $\R^N$ ($\phi$ simply extends the radius) and $C^2$ domains in general. In \cref{sec:Gc}, under \Cref{Ocond}, we provide proofs for the $\Gamma$-liminf and limsup inequalities. The next corollary summarizes these proofs.

\begin{corollary}\label{corol:FkGammaconv}
    Suppose that $\alpha > 0$, $\epsilon_k \downto 0$, and that \cref{Ocond} holds. Then $\bar F_k: \mathscr{B}(\Omega) \times \mathscr{B}(\Omega)  \to \bar \R$,
    defined by \eqref{eq:Fk1app} $\Gamma$-converge to $F: \mathscr{B}(\Omega)  \times \mathscr{B}(\Omega)  \to \bar \R$ defined by \eqref{eq:defF}.
\end{corollary}

\begin{proof}
    The $\Gamma$-liminf inequality \eqref{eq:Gamma1} is given by \cref{lemma:liminf} and the $\Gamma$-limsup inequality \eqref{eq:Gamma2} is given by \cref{corol:GSBV}.
    This finishes the proof.
\end{proof}

%%%%%%%%%%%%%%%%%%%%%%%%%%%%%%%%%%%%%%%%%%%%%%%
\subsection{Convergence of solutions}
%%%%%%%%%%%%%%%%%%%%%%%%%%%%%%%%%%%%%%%%%%%%%%%
Finally, we show the convergence of solutions of \eqref{eq:minFGapp} to solutions of \eqref{eq:minFG}. We do this using the continuity of the currents and $\Gamma$-convergence of $\bar F_k$ to $F$:
\begin{theorem}
    \label{thm:measure-compactness}
    Suppose that $\alpha > 0$, and that \cref{Ocond} holds. Let $\lbrace (\bar \gamma_k,\bar z_k) \rbrace \subset \bar D_{k,N}(\Omega)$ be a sequence of solutions to \eqref{eq:minFGapp} so that $\epsilon_{k} \downto  0$. Then there exists a subsequence $ (\bar \gamma_{k_l},\bar z_{k_l})  \to (\bar \gamma,0)$ in measure so that $(\bar \gamma, 0)$ solves \eqref{eq:minFG}.
\end{theorem}
\begin{proof}
    By \cite[Theorem 3.6]{ambrosio1990approximation}, any sequence $\lbrace (\gamma_k, z_k) \rbrace \subset D_{k,N}(\Omega)$ such that
    \begin{equation}\label{eq:subseq}
        F_k(\gamma_k, z_k) + \int_\Omega \abs{\gamma_k}^2\ddd\Lmeas \le C < \infty
    \end{equation}
    has a subsequence $ ( \gamma_{k_l}, z_{k_l}) \to (\gamma,0) $ in measure. Since each $(\bar \gamma_k,\bar z_k)$ is a solution to \eqref{eq:minFGapp}, 
    $$
        G(\bar \gamma_k) + \bar F_k(\bar \gamma_k,\bar z_k) \le \bar G(\gamma_m) + \bar F_k(\gamma_m,0) =  C_1 < \infty.
    $$ 
    Since $\int_\Omega \abs{\bar \gamma_k}^2\ddd\Lmeas \le \int_\Omega \abs{\gamma_M}^2\ddd\Lmeas \defeq C_2 < \infty$, \eqref{eq:subseq} holds for $(\bar \gamma_k,\bar z_k)$ with $C=C_1 + C_2$
    and thus by \cite[Theorem 3.6]{ambrosio1990approximation}, $\lbrace (\bar \gamma_k,\bar z_k) \rbrace$ has a subsequence, unrelabelled, with $(\bar \gamma_{k},\bar z_{k})  \to (\gamma,0) $ in measure. 

    Similarly to \cite{ambrosio1990approximation}, next we show that 
    \begin{equation}\label{eq:ggg}
        F(\gamma,0) + G(\gamma) \ge F(\bar \gamma,0) + G(\bar \gamma),\quad \text{for all } \gamma \in \SBV(\Omega).
    \end{equation}
    Assuming $F(\gamma,0)$ is finite, $\gamma \in [\gamma_m,\gamma_M]$. Let $ (\gamma_{k},z_{k}) \to (\gamma,0)$ be the reconstruction sequence of \cref{corol:FkGammaconv} for the same $\epsilon_{k}\downto 0$ so that $\gamma_{k} \in [\gamma_m,\gamma_M]$. Then due to $\Gamma$-liminf and limsup inequalities, $\lim_{k\to\infty} \bar F_{k}(\gamma_{k},z_{k}) = F(\gamma,0)$. Since \cref{Ocond} holds and $\gamma,\gamma_{k} \in [\gamma_m,\gamma_M]$, by \cref{thm:cim},
    $
        G(\gamma) = \lim\limits_{k \to \infty } G(\gamma_{k}).
    $
    Combining these yields 
    \begin{equation}\label{eq:mine1}
        F(\gamma,0) + G(\gamma) 
        =  \lim\limits_{k \to \infty } \bar F_{k}(\gamma_{k},z_{k}) + G(\gamma_{k}).
    \end{equation}
    Since $\lbrace (\bar \gamma_{k},\bar z_{k}) \rbrace_{k\in\N}$ solves \eqref{eq:minFGapp}, we have
    \begin{equation}\label{eq:mine3}
        \liminf\limits_{k \to \infty } \bar F_{k}(\gamma_{k},z_{k}) + G(\gamma_{k}) 
        \ge \liminf\limits_{k \to \infty } \bar F_{k}(\bar \gamma_{k},\bar z_{k}) + G(\bar \gamma_{k}) .
    \end{equation}
    Clearly $\bar F_k(\bar \gamma_{k}, \bar z_k)$ is finite, meaning that $\bar \gamma_{k} \in [\gamma_m,\gamma_M]$ and $\limsup_{k\to \infty} \bar \gamma_{k}(x) \le \gamma_M$ and $\liminf_{k\to \infty} \bar \gamma_{k}(x) \ge \gamma_m$ a.e. imply that $\bar \gamma \in [\gamma_m,\gamma_M]$ so that $G(\bar \gamma_{k}) \to G(\bar \gamma)$ again by \cref{thm:cim}. In addition, by \Cref{corol:FkGammaconv}, it holds that $\liminf\limits_{k \to \infty } \bar F_{k}(\bar \gamma_{k},\bar z_{k}) \ge F(\gamma,0)$, hence
    \begin{equation}\label{eq:mine4}
        \liminf\limits_{k \to \infty } \bar F_{{k}}(\bar \gamma_{{k}}) + G(\bar \gamma_{{k}}) 
        \ge  F(\bar \gamma,0) + G(\bar \gamma).
    \end{equation}
    Combining \cref{eq:mine1,eq:mine3,eq:mine4} establishes \eqref{eq:ggg}, consequently proving the claim.
\end{proof}

\section{Numerical simulations and experimental study}
\label{sec:numerical}
In this section, we numerically evaluate the performance of $F_\lambda$ regularization in EIT with simulated and experimental measurement data. We consider seven test cases with synthetic data (Cases 1-7), all corresponding to different conductivity distributions. In Case 8, we use experimental data.

In the numerical scheme, we employ
a two-dimensional domain $\Omega \subset \R^2$, and use Galerkin finite element (FE) approximation $G_h$ of the data term $G$. We define it by 
$$
    G_h(\gamma) \defeq  \tfrac{1}{2a}\norm{W(I_h(\gamma) - \Imeas)}_{2}^2%, & \gamma_m \le \gamma \le %\gamma_M \\ +\infty, & \mbox{otherwise, } \end{cases}\begin{cases}
$$
with $I_h(\gamma) =(I^1_h(\gamma),\dots,I^{P_2}_h(\gamma),)$ consisting of the electric current vectors $I_h^j(\gamma)$ of $w_h(\gamma)\defeq (u_h^j(\gamma),I_h^j(\gamma)) \in \Hs_h(\Omega)$, corresponding to multiple potential arrangements $U^j$. The space $\Hs_h(\Omega)$ here denotes a finite elements space $\Hs_h(\Omega) \subset \Hs(\Omega)$ spanned by piecewise linear basis functions $\phi_i$. In addition to $u^j$, we also discretize the conductivity $\gamma$ and the control parameter $z$ through the same basis so that 
$$
u^j(x) = \sum_{i=1}^{N_n} u_i^j\phi(x), \quad \gamma(x) = \sum_{i=1}^{N_n} \gamma_i\phi_i(x),\quad\text{and}\quad z(x) = \sum_{i=1}^{N_n} z_i\phi_i(x).
$$
To simplify the computations, instead of the approximating functions $F_k$ of \eqref{eq:Fk1app}, we use 
\begin{equation}\label{eq:Fe}
    F_\lambda(\gamma,z) 
    = \int_\Omega \left[\lambda \abs{\nabla z}^2 + z^2\abs{\nabla \gamma}^2 + \frac{\alpha^2(z-1)^2}{4\lambda}\right] + \delta_{[\epsilon_\lambda,1]}(z) + \delta_{[\gamma_m,\gamma_M]}(\gamma) \ddd\Lmeas
\end{equation}
for some $(0, 1) \ni \epsilon_\lambda \downto 0$. The function $\delta$ is the indicator function.
This functional with $\epsilon=0$ was suggested in a remark in \cite{ambrosio1990approximation} as an alternative to $F_k$. Alternatively to the term $o_\lambda \abs{\nabla \gamma}^2$ in \cite{ambrosio1992approximation}, we use $\epsilon_\lambda>0$ to ensure the existence of solutions.
The $\Gamma$-convergence of $F_\lambda$ to $F$ can be proved performing similar modifications to the proof of \cref{ssec:Gconv} as was done to the proof of \cite{ambrosio1990approximation} in \cite{ambrosio1992approximation}. In numerical practise, $\epsilon_\lambda$ is unnecessary, so we set it to zero.

Since $\nabla \phi$ is constant within each element of the FE-mesh, $F_\lambda$ admits a very simple form in the chosen basis.
Denoting $E_i \subset \Omega$ as the element i of the FE mesh, $A_i$ as the area of $E_i$, and $\phi_k$ and $z_k$, $k=1,2,3$, as the three basis functions within $E_i$, we have that $\abs{\nabla z}^2 \equiv g_{z,i}\defeq (z_1 \partial_x \phi_1 + z_2 \partial_x \phi_2 + z_3 \partial_x \phi_3 )^2 + (z_1 \partial_y \phi_1 + z_2 \partial_y \phi_2 + z_3 \partial_y \phi_3 )^2$ and since the partial derivatives of $\phi_i$ are constant within an element,
$$
\int_{E_i} \abs{\nabla z}^2 \ddd\Lmeas = A_ig_{z,i}.
$$
Further, simple computations show that
\begin{equation}\notag
    \int_{E_i} z^2\abs{\nabla \gamma}^2 \ddd\Lmeas 
    = A_i g_{\gamma,i}( z_1^2 + z_2^2 + z_3^2 + z_1z_2 + z_1z_3  + z_2z_3)/12
\end{equation}
and
\begin{equation}\notag
    \int_{E_i} (z-1)^2 \ddd\Lmeas 
    = A_i((z_1^2 + z_2^2 + z_3^2 + z_1z_2 + z_1z_3  + z_2z_3)/12 - (z_1 + z_2 + z_3)/3+ 1/2).
\end{equation}
Next we will lay out the plan to minimize $F_\lambda(\gamma) + G_h(\gamma)$.

\subsection{Solving the discretized problem}

Our task therefore is to solve the problem
\begin{equation}\label{eq:minDisc}
    \min\limits_{\gamma,z} J(\gamma,z),\quad J(\gamma,z) \defeq F_\lambda(\gamma,z) + G_h(\gamma) + \int_\Omega \delta_{[\gamma_m,\gamma_M]}(\gamma) +\delta_{[0,1]}(z) \ddd\Lmeas
\end{equation}
in the above-constructed finite element spaces.
We include the regularization parameter $a$ into $W$ by taking $\WM = \sqrt{a}\tilde{ \WM}$.

We base our iterative approach on the Relaxed Inexact Proximal Gauss-Newton (RIPGN) algorithm of \cite{ripgn}. This method and Gauss–Newton type methods in general are based on the iterative solution of convex subproblems obtained through the linearization of operators in the original problem. We will deviate from this slightly, only linearizing in \eqref{eq:minDisc} the operator $I_h$ within $G_h$, but retaining the non-convexity of $F_\lambda$. Correspondingly, we also do not perform the relaxation step of the RIPGN on the control variable $z$, only on $\gamma$. These choices are based on practical numerical experience, but make the convergence theory of \cite{ripgn} inapplicable.

\begin{algorithm}[t!]
    \caption{Relaxed inexact proximal Gauss--Newton method (RIPGN) adapted to the problem \eqref{eq:minDisc}.}
    \label{alg:gn-overrelax}
    \begin{algorithmic}[1]
        \Require Proximal parameter $\beta \ge 0$ and relaxation parameter $w > 0$.
        \State Choose an initial iterate $(\gamma^1,z^1) \in \Dom J$.
        \ForAll{$k \ge 1$}
        \State Find an approximate solution $(\tilde \gamma^k, \tilde z^k)$ to \eqref{eq:minJprox}
        \State Update $\gamma^{k+1} \defeq (1-w)\gamma^k + w\tilde \gamma^k$ and $z^{k+1} \defeq \tilde z^k$
        \EndFor
    \end{algorithmic}
\end{algorithm}

\begin{algorithm}[t!]
    \caption{Non-linear primal-dual proximal splitting for solving \eqref{eq:minJprox}. We denote the convex conjugate of $F$ by $F^*$ and $q = (q_1,q_2))$ such that $q_1$ corresponds to $\gamma$ and $q_2$ corresponds to $z$.}
    \label{alg:alg-blockcp}
    \begin{algorithmic}[1]
        \Require Convex, proper, lower semicontinuous $F: X \to \extR$, $H: Y \to \extR$, and operator $K\in C^2(X; Y)$.
        \State Choose step length parameters $t^j, s^j >0$ satisfying $s^j < 1/(t^j L^2)$ for some upper bounds $L \ge \sup_{k=1,\dots, j} \norm{ \nabla K(q^k)^T}$.
        \State Choose initial iterates $q^0$ and $y^0$.
        \ForAll{}{ $j \ge 0$ \textbf{until} a stopping criterion is satisfied}
        \State $
            q^{j+1} \defeq \prox{t^i}{H}{
            q^i - t^j ([\nabla K(q_j)] y)
            }
            $
        \State
            $
            \bar q^{j+1}
            \defeq
            2q^{j+1}-q^j
            $
        \State
            $
            y^{j+1} \defeq \prox{s^j}{F^*}{
            y^j + s^j K (\bar q^{j+1})
            }
            $
        \EndFor
    \end{algorithmic}
\end{algorithm}

Specifically, in \cref{alg:gn-overrelax} we replace $G_h$ on each iteration $k$ by
$$
    G_h^k(\gamma) 
    = \tfrac{1}{2a}\norm{\WM[I_h(\gamma^k)-\mathscr I + \nabla I_h(\gamma^k)^T(\gamma -\gamma^k)] }^2_{2}
$$ 
and with a suitable algorithm solve the proximal-penalized partially linearised problem
\begin{equation}\label{eq:minJprox}
    \min\limits_{\gamma,z} F_\lambda(\gamma,z) + G_h^k(\gamma) +  \delta_{[\gamma_m,\gamma_M]}(\gamma) +\delta_{[0,1]}(z) + \beta\norm{\gamma - \gamma^k}_2^2.
\end{equation}
Since $F_\lambda$ is still nonconvex, we use the Nonlinear Primal-Dual Proximal Splitting (NL-PDPS) of \cite{tuomov-nlpdhgm}.
This algorithm applies to problems of the general form
\begin{equation}
    \label{eq:nlpdps-problem}
    \min_q F(K(q)) +  H(q)
\end{equation}
where the operator $K$ is possibly non-linear, but $F$ and $H$ are convex but possibly nonsmooth. \Cref{alg:alg-blockcp} writes out this method with dynamic adaptation of the dual step length parameters $s^j$ and $t^j$ to the step length conditions.
To present the problem \eqref{eq:minJprox} in the form \eqref{eq:nlpdps-problem}, we write
$$
    G_h^k(\gamma) = \tfrac{1}{2}\norm{K_1^k\gamma - b^k}_2^2 = F_1(K_1^k\gamma), 
$$
where $K_1^k = \WM \nabla  I_h(\gamma^k)^T$, $b^k = K^k_1\gamma^k - \WM(I_h(\gamma^k) - \Imeas)$, and $F_1(y) = \tfrac{1}{2}\norm{y - b^k}_2^2$. Likewise, we write
$$
    F_\lambda(\gamma,z) = F_2(K_2(\gamma,z))
\quad\text{with}\quad
    K_2(\gamma,z)_i 
    =  \int_{E_i} \left[\lambda \abs{\nabla z}^2 + z^2\abs{\nabla \gamma}^2 + \frac{\alpha^2(z-1)^2}{4\lambda}\right] \ddd\Lmeas,\quad (i=1,\dots,N_e)
$$
for the components of non-linear function $K_2 : \R^{2N_n} \to \R^{N_e}$, and $F_2(y) =  \norm{y}_1$. Finally, we set
$$
    H(\gamma,z) =  \delta_{[\gamma_m,\gamma_M]}(\gamma)+\delta_{[0,1]}(z).
$$
Thus, defining $K^k(\gamma,z)=(K_1^k\gamma, K_2(\gamma,z))$ and $F(y_1,y_2) = F_1(y_1) + F_2(y_2)$ \eqref{eq:minJprox} reads
\[
    \min\limits_{\gamma,z} F(K^k(\gamma,z)) + H(\gamma,z).
\]

\subsection{Studies with synthetic data}
\label{sec:naes}

\begin{table}
    \caption{Description of the numerical and experimental test cases. Here, $\gamma_\mathrm{bg}$ is the background conductivity and $\gamma_\mathrm{ci}$ and $\gamma_\mathrm{ri}$  refer respectively to the non-smooth conductive and resistive inclusions placed on the smooth/constant background. Note that Cases 1-7 use simulated data. Inclusion $\gamma_\mathrm{ci}$ and $\gamma_\mathrm{ri}$ in Case 8 describe cross-section of the objects. }
    \label{table:Cases}
    \centering
    \footnotesize%\small
    \begin{tabular}{ c l l l l}
         & \hspace{1cm}$\gamma_{\mathrm{bg}}$ &\hspace{1.5cm}$\gamma_{\mathrm{ci}}$ &\hspace{1.5cm}$\gamma_{\mathrm{ri}}$ &\hspace{0.75cm}Purpose of the test \\ 
        \hline
        \hline
        Case 1& Smoothly varying& Circular $\gamma = 10$ S &Square-shaped $\gamma = 10^{-4}$ S& \tabitem Test effects of $\alpha$ and $\lambda$. \\
        &  &  &  & \tabitem Comparison with $\TV$\\
        &  &  &  & \;\;\; and $L^2$. \\
        %\hline
       % Case 2& Smoothly varying & Circular $\gamma = 10$ S & Narrow $\gamma = 10^{-4}$ S&\tabitem Comparison with $\TV$. \\
        \hline
        Case 2& Smoothly varying & Stadium-shaped $\gamma = 7.5$ S & None & \tabitem Test with fixed $\alpha$ and $\lambda$.\\
        \hline
        Case 3& Smoothly varying & None & Stadium-shaped $\gamma = 10^{-4}$ S & \tabitem Test with fixed $\alpha$ and $\lambda$. \\
        \hline
        Case 4&  Smoothly varying & Triangular $\gamma = 8.5$ S & Triangular $\gamma = 10^{-3}$ S  & \tabitem Test with fixed $\alpha$ and $\lambda$. \\
        \hline
        Case 5& Constant $\gamma = 1$ S & None & Circular $\gamma = 10^{-4}$ S & \tabitem Test with fixed $\alpha$ and $\lambda$. \\
        \hline
        Case 6& Constant $\gamma = 1$ S & Circular $\gamma = 7.5$ S & None  & \tabitem Test with fixed $\alpha$ and $\lambda$.\\ 
        \hline
        Case 7& Smoothly varying & None & None & \tabitem Test with fixed $\alpha$ and $\lambda$. \\ 
        \hline
        Case 8& Water & Circular steel & Circular plastic & \tabitem Experimental tests \\ 
    \end{tabular} 
\end{table}%\todo[inline]{Better explanation for Cases 3-8 in the Table 1?} 

In Case 1, we investigate the effects of the control parameter $\lambda$ and the regularization parameter $\alpha$, and further, we compare the results with these parameters to smooth regularization $F_\nabla(\gamma)=\norm{ \abs{\nabla \gamma}}_2^2$ and TV regularized solutions. In Cases 2-7, with fixed $\lambda$ and $\alpha$, which are chosen based on results in Case 1, we will test Mumford–Shah regularization with multiple sets of measurement data generated with varying conductivities, comparing it against $\TV$ regularization. The Cases 1-7 use simulated data. %

\Cref{table:Cases} describes the selected true conductivity distribution in each test case as well as the purpose of each test. In the simulated cases, the data represents electrical measurements collected on the boundary of a disk shaped domain $\Omega$ (see \cref{fig:Disk_true}). The radius of the disk is $15$ cm and it has 16 evenly placed electrodes on the border. The length of an electrode is 1/32 of the border length. We draw the smooth background conductivity from a Gaussian distribution with mean of $1$ S and the so-called squared distance covariance matrix with correlation length of 12 cm \cite{kaipio2006statistical,ripgn}. Note that in \cref{fig:Disk_true}, we use a highly nonlinear color scale (see \Cref{table:CS} for exact values) --- this helps to illustrate the smooth variations of the conductivity in the background which vary in the range of $1$ S $\pm 0.2 $ S. With a linear color scale the background would have almost constant color, since the inclusions of low and high conductivity are $10^{-4}$ and $10$ S, respectively. %
\definecolor{legend0}{rgb}{0,0,0}
\definecolor{legend1}{rgb}{0.5,0,0}
\definecolor{legend2}{rgb}{1.0,0.7,0}
\definecolor{legend3}{rgb}{1.0,1.0,1.0}
\definecolor{legend4}{rgb}{0.0,1.0,1.0}

\begin{table}
    \caption{Description of the non-linear color scale used in \Cref{fig:Disk_true,fig:Case0,fig:Case1,fig:Case3}. }
    \label{table:CS}
    \centering
    \footnotesize%\small
    \begin{tabular}{ r r r r r r}
        {Conductivity (S)} \vline& {$0$} & {$0.8$} &{ $1.0$} & { $1.2$} & {$10$ }\\
        {Color (R,G,B) } \vline& {$0,0,0$ }& {$0.5, 0, 0$} &{ $1, 0.7, 0$}&{ $1, 1, 1$}&{ $0, 1, 1$}
        \\
        \vline&
        \tikz \fill [legend0] (0,0) rectangle (1,0.3);&
        \tikz \fill [legend1] (0,0) rectangle (1,0.3);&
        \tikz \fill [legend2] (0,0) rectangle (1,0.3);&
        \tikz \draw [fill=legend3] (0,0) rectangle (1,0.3);&
        \tikz \fill [legend4] (0,0) rectangle (1,0.3);\\
    \end{tabular} 
\end{table}%\todo[inline]{Better explanation for Cases 3-8 in the Table 1?} 

To simulate the EIT measurements, we set the contact impedances to $\zeta_i = 10^{-5}$ $\Omega$, and use voltage patterns $U^j$, $j=1,2,\dots, 16$, where we set electrode $j$ to a known potential $(U^j)_j=1$ V and ground the others by setting $(U^j)_i = 0$ V for $i \neq j$. From the simulated current vectors $\Imeas^j$, we exclude the $j$'th current, the injection current, as the EIT measurement devices do not usually measure this current. 
Further, in these simulations we use a mesh that has 5039 nodes, and we add Gaussian noise with $10^{-4}\abs{(\mathscr I^j)_i}$ std to each simulated measurement $(\mathscr I^j)_i$. Note that with this setup, $\mathscr I= \left((\mathscr I^1)_2,(\mathscr I^1)_3,\dots, (\mathscr I^{16})_{15}\right)$, i.e., $P_1 = P_2 = 16$ while $\M=P_2(P_1-1) = 240$. 

For the inversion, we use sparser mesh with 2917 nodes. We start each reconstruction iteration from the best homogenous estimate \cite{jarvenpaa1996} $\gamma^1=\gamma_{\text{hmg}}$ of the conductivity and for the control variable we set $z^1 = \boldsymbol{1}$, where $\boldsymbol{1} \in \R^{N_n}$ is a vector of ones. The latter only has effect on the initial guess of $z$ at the first outer iteration. We use the following parameter configuration for the RIPGN and the NL-PDPS:
\begin{itemize}
    \item For the primal step length parameter $t^j$ of NL-PDPS we set $t^j=0.01$ for all $j \ge 0$.
    \item For the dual step length parameter $s^i$ of NL-PDPS we set $s^i=1/(2t^j(\norm{K_1}^2 + (1.2\norm{\nabla K_2(q^i)^T}_F)^2))$, where $\norm{\cdot}_F$ is the Frobenius norm. We only update $s^i$ between hundred iterations. 
    \item In NL-PDPS, we use initial iterates $q_1^1 = \gamma^k$, $q_2^1 = \tfrac{3}{4} \boldsymbol{1} + \tfrac{1}{4}z^k$, and $y=0$; we assume that $\tilde \gamma^k$ is close to $\gamma^k$ and the sharp edges in $\tilde \gamma^k$ slightly resemble those of $\gamma^k$. 
    \item For RIPGN, we set the proximal parameter to $\beta = 0.01$ and the relaxation parameter $w=0.1$. We use these to reduce the length of the steps, i.e. $\norm{\gamma_{k+1}-\gamma_k}_2$, to avoid using line search.
\end{itemize}
For further explanations of these choices, we refer to \cite{ripgn}. 

We do not assume to know neither the maximum nor the minimum conductivities precisely, and thus we set the maximum conductivity to an arbitrary large number $\gamma_M = 10^{10}$ S, and the minimum conductivity to $\gamma_m = 10^{-5} \gamma_{\text{hmg}}$ S. However, we do assume that we know the distribution of the noise approximately, meaning that we compute the weighting matrix from the measurements by setting $\tilde{\WM}_{i,j} = 200/\abs{\mathscr I_i}\delta_{i,j}$, where $\mathscr I_i$ is the $i$'th component of the measurement vector $\mathscr I$, and $\delta_{i,j}$ is the Kronecker delta. %Clearly, this assumption of the std is off by the factor of $2$.

For a detailed explanation on the finite element approximation of $(u,I)$ and how the $\TV$ and $F_\nabla$ regularized solutions are solved using RIPGN, we refer to \cite[Section 4]{ripgn}.
\begin{figure}[!tbp]
    \centering
    \begin{minipage}{0.650\textwidth}  
        \centering
        \subfloatrecoi{Case 1}{0.42\textwidth}{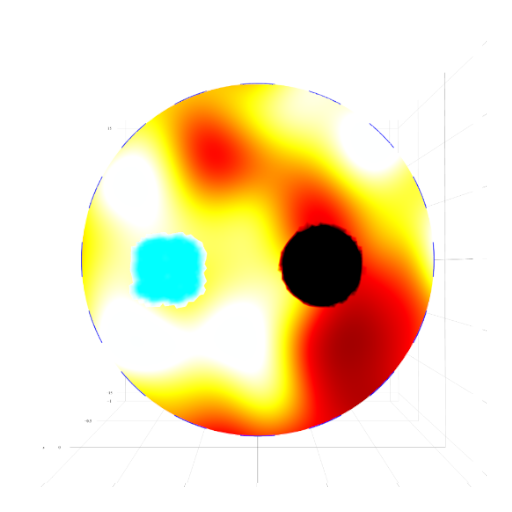}  
    \end{minipage}
    \subfloatcolorbar{-0.5cm}{-2.5cm}{0.70cm}{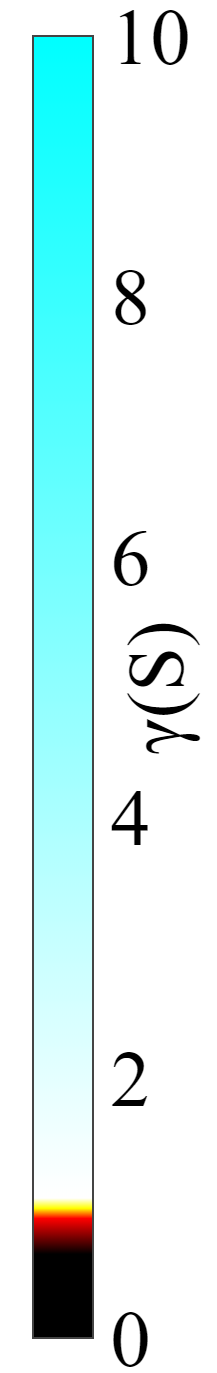}%
    \caption{Conductivities used to simulate the measurement data in Cases 1.}
    \label{fig:Disk_true}
\end{figure}
\subsubsection{Results: Synthetic data}\label{sssec:ressynth}
\begin{figure*}[!t]%
    \center
    \rotatebox{90}{\hspace{1.4cm}\scriptsize{$G_h(\gamma_k) + F_\lambda(\gamma_k)$}}
    \includegraphics[height=0.16\paperheight, trim=0 20 0 50,clip]{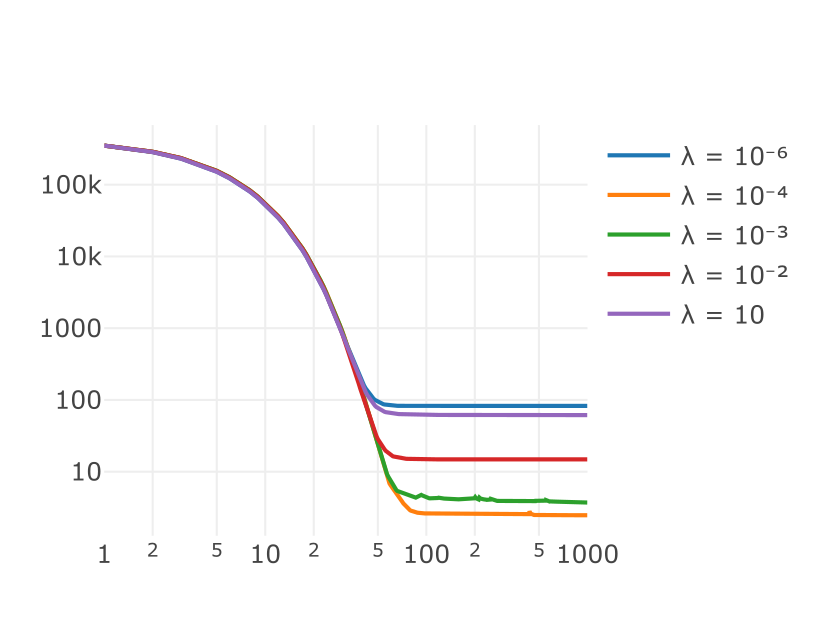}\\
    \scriptsize{\hspace{-0.45cm}$k$}
    \caption{Case 1. Convergence of the objective function $G_h(\gamma) + F_\lambda(\gamma)$ with varying controls $\lambda$.}
    \label{fig:Case0conv}%
\end{figure*}%
\captionsetup[subfloat]{position=top,labelformat=empty}
\begin{figure*}[!t]%
    \centering%
    \subfloatcolorbar{0.0cm}{-1.9cm}{0.044\textwidth}{"Figures/Case0/CBzs".png}%   
    \begin{minipage}{0.90\textwidth}%
    \centering%
        \rotatebox{90}{\hspace{-3.2cm} $z$ \hspace{1.7cm} $\gamma$}
        \subcstageb{$\lambda = 10^{-6}$}{sfig:c2s1}{0.143\linewidth}{"Figures/Case0/Disk16_Case0_eps1".png}{"Figures/Case0/Disk16_Case0_eps1z".png}{80}{80}{80}{80}%
        \subcstageb{$\lambda = 10^{-4}$}{sfig:c2s2}{0.143\linewidth}{"Figures/Case0/Disk16_Case0_eps2".png}{"Figures/Case0/Disk16_Case0_eps2z".png}{80}{80}{80}{80}%
        \subcstageb{$\lambda = 10^{-3}$}{sfig:c2s3}{0.143\linewidth}{"Figures/Case0/Disk16_Case0_eps3".png}{"Figures/Case0/Disk16_Case0_eps3z".png}{80}{80}{80}{80}%
        \subcstageb{$\lambda = 10^{-2}$}{sfig:c2s4}{0.143\linewidth}{"Figures/Case0/Disk16_Case0_eps4".png}{"Figures/Case0/Disk16_Case0_eps4z".png}{80}{80}{80}{80}%
        \subcstageb{$\lambda = 10$}{sfig:c2s5}{0.143\linewidth}{"Figures/Case0/Disk16_Case0_eps5".png}{"Figures/Case0/Disk16_Case0_eps5z".png}{80}{80}{80}{80}%
        \subcstageb{$F_\nabla,$ $ a = 1$}{sfig:c2s6}{0.143\linewidth}{"Figures/Case1/Disk16_Case1_L2_6".png}{"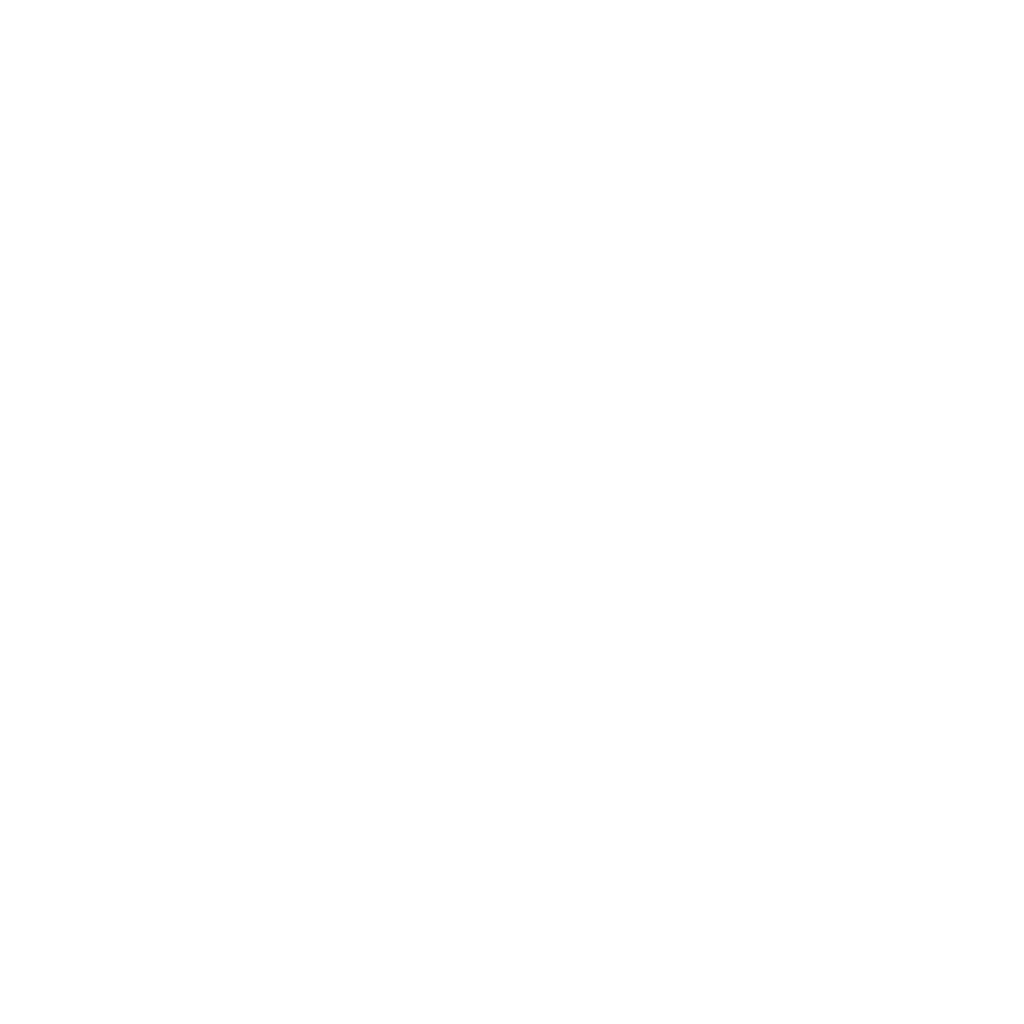"}{130}{130}{130}{130}%       
        %\subcstageb{Stage 4.}{sfig:c2s4}{0.18\linewidth}{"Figures/Case0/Chamber5_w1s".png}{"Figures/Case0/Chamber5_w1".png}{200}{50}{200}{200}%                  
    \end{minipage}% 
    \subfloatcolorbar{0.0cm}{-1.9cm}{0.044\textwidth}{"Figures/Case0/CBs".png}%    
    \caption{Case 1. Effects of the control parameter $\lambda$. Top row, first 5 columns: $F_\lambda$ regularized EIT reconstructions with control parameter values $\lambda = 10^{-6}, 10^{-4}, 10^{-3}, 10^{-2},10$. The regularization parameters are $\alpha=10^{-2}$ and $a= 1$. The last column shows $F_\nabla$ regularized solution with $a = 1$ for comparison. Bottom row: the respective reconstructions of the variable $z$. The color scale on the left corresponds to values of $z$ and the right one corresponds to the values of $\gamma$.}%ϵ = 1e-6,ϵ = 1.0e-4,,ϵ = 1e-3, ϵ = 1e-2, ϵ = 10.0, α = 1.0e-2
    \label{fig:Case0}%
\end{figure*}%
\captionsetup[subfloat]{position=bottom}

\Cref{fig:Case0conv} shows convergence of the objective function in Case 1 with the tested controls $\lambda$. All reconstructions converge properly; the values of the objective functionals drop to less that 1/1000 of the initial value. Depending on the case and on the chosen parameters, the algorithm took from around 100 to 1000 iterations to stagnate\footnote{Ryzen 5950X with 32 GB of 3600 Mhz (16-19-19-39) DDR4 RAM, an iteration took 13.58 seconds to compute on average.}. Very large or small $\lambda$ yield worse fit for the data, since they do not allow the sharp edges to form properly in $\gamma$. Next, we will inspect this more closely.

\cref{fig:Case0} shows the M-S regularized reconstructions with six values of the control parameter $\lambda$ and fixed $\alpha = 10^{-2}$ in Case 1. We fix $a=1$ for the $F_\lambda$ regularizer throughout all the tests. For comparison, the last column shows a solution computed with $F_\nabla$ regularization.

Clearly, \cref{fig:Case0} shows that the solutions tend towards the $F_\nabla$ regularized solution if $\lambda$ is either very large or very small. The former behavior is obvious, since increasing $\lambda$ increases the weight of the $F_\nabla$ term of $z$ in \eqref{eq:Fe}, meaning that less spatial variation is expected in $z$ and thus $z \approx a = 1$ everywhere and no clear edges are formed in $\gamma$. The latter behavior, on the other hand, is caused by the discretization of $z$; recall that for a reconstruction sequence, $z \to 1$ in measure as $\lambda\downto 0$, and thus the area in which $z < 1$ tends to zero\footnote{For functional \eqref{eq:Fk1app}, this is equivalent to $k\to \infty$ and $z\to 0$ in measure, and the area where $z > 0$ tends to zero.}. %, as discussed in \cref{ssec:Gconv},
However, since $z$ is discretized, the smallest area in which it is possible to set $z < 1$ is fixed, and thus, since this $z$ is a solution to \eqref{eq:minDisc}, increasing $1/\lambda^2$ over a certain threshold has to, instead, increase pointwise value of $z$ everywhere. This clearly means that the penalty of $\abs{\nabla \gamma}^2$ will tend to $1$ everywhere reverting the reconstruction back to the $F_\nabla$ regularized solution with $a =1$. This observation suggests that $\lambda$ should be chosen according to the discretization (i.e. the triangles sizes) of $z$. Based this test, we fix $\lambda = 10^{-3}$ for the following tests. 
\captionsetup[subfloat]{position=top,labelformat=empty}
\begin{figure*}[!t]
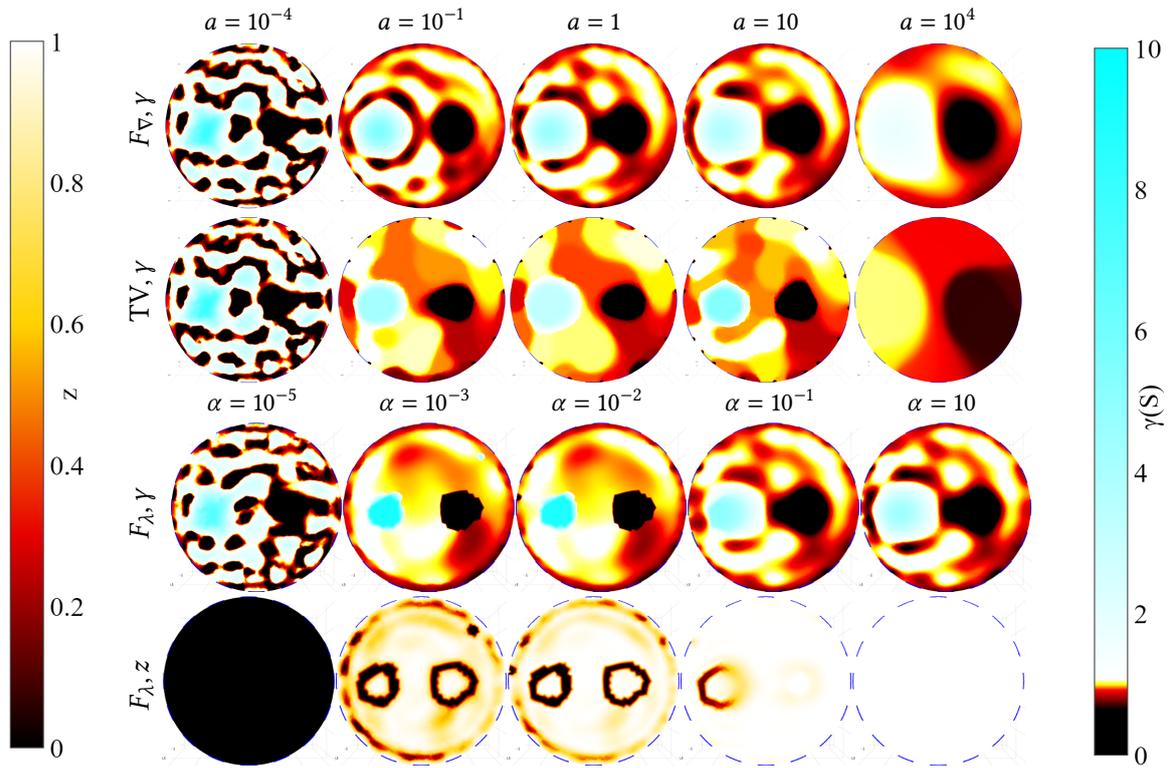
%
    \centering%
    \subfloatcolorbar{0.0cm}{-4.8cm}{0.099\textwidth}{"Figures/Case0/CBz".png}%  
    \begin{minipage}{0.8\textwidth}%
    \centering%
        \rotatebox{90}{\hspace{-4.2cm}$F_\lambda,z$ \hspace{1.4cm} $F_\lambda,\gamma$ \hspace{2cm}  TV$,\gamma$ \hspace{1.3cm} $F_\nabla,\gamma$}
        \subcstagee{$a= 10^{-4}$}{$\alpha= 10^{-5}$}{"Figures/Case1/Disk16_Case1_L2_2".png}{"Figures/Case1/Disk16_TV_2".png}{"Figures/Case1/Disk16_Case1_1".png}{"Figures/Case1/Disk16_Case1_1z".png}%
        \subcstagee{$a= 10^{-1}$}{$\alpha= 10^{-3}$}{"Figures/Case1/Disk16_Case1_L2_5".png}{"Figures/Case1/Disk16_TV_5".png}{"Figures/Case1/Disk16_Case1_2".png}{"Figures/Case1/Disk16_Case1_2z".png}%
        \subcstagee{$a= 1$}{$\alpha= 10^{-2}$}{"Figures/Case1/Disk16_Case1_L2_6".png}{"Figures/Case1/Disk16_TV_6".png}{"Figures/Case1/Disk16_Case1_3".png}{"Figures/Case1/Disk16_Case1_3z".png}%
        \subcstagee{$a= 10$}{$\alpha= 10^{-1}$}{"Figures/Case1/Disk16_Case1_L2_7".png}{"Figures/Case1/Disk16_TV_7".png}{"Figures/Case1/Disk16_Case1_4".png}{"Figures/Case1/Disk16_Case1_4z".png}%
        \subcstagee{$a= 10^{4}$}{$\alpha= 10$}{"Figures/Case1/Disk16_Case1_L2_10".png}{"Figures/Case1/Disk16_TV_10".png}{"Figures/Case1/Disk16_Case1_5".png}{"Figures/Case1/Disk16_Case1_5z".png}%
        %\subcstaged{Parameter 2}{0.30\linewidth}{"Figures/Case1/Disk16_TV_2".png}{"Figures/Case0/Disk16_Case1_eps3".png}{"Figures/Case1/Disk16_MS_2z".png}{130}{125}{150}{130}%
        %\subcstaged{Parameter 3}{0.30\linewidth}{"Figures/Case1/Disk16_TV_4".png}{"Figures/Case1/Disk16_MS_5".png}{"Figures/Case1/Disk16_MS_5z".png}{130}{125}{150}{130}%
        %\subcstageb{Stage 4.}{sfig:c2s4}{0.18\linewidth}{"Figures/Case1/Chamber5_w1s".png}{"Figures/Case1/Chamber5_w1".png}{200}{50}{200}{200}%                  
    \end{minipage}% 
    \subfloatcolorbar{0.0cm}{-4.7cm}{0.099\textwidth}{"Figures/Case0/CB".png}%    
    \caption{Case 1. Effects of the regularisation parameter $\alpha$. First row: $F_\nabla$ regularised solution. Second row: $TV$ regularised solution. Third and forth row: $F_{\lambda}$ regularisation based solution of $\gamma$ and $z$. Each column corresponds to a different regularisation parameter values and the control parameter is $\lambda = 10^{-3}$. Note that the regularisation parameter in TV and $F_\nabla$ is the weight of the gradient term, while on the other, in $F_\lambda$ solution it is jump set regularization parameter.}
    \label{fig:Case1}%
\end{figure*}%
\captionsetup[subfloat]{position=bottom}
%L2 α = 1e-7; 1e-4; 1e-3; 1e-2; 1e-1; 1.0; 1e1; 1e2; 1e3; 1e4
%TV α = 1e-7; 1e-4; 1e-3; 1e-2; 1e-1; 1.0; 1e1; 1e2; 1e4; 1e3
%MS α = 1.0e-5 1.0e-3,α = 1.0e-2,α = 1.0e-1 α = 1e1

\Cref{fig:Case1} shows comparison between $F_\lambda$, TV, and $F_\nabla$ regularizations with varying regularization parameters $\alpha$ and $a$. By \Cref{fig:Case1}, as the regularization parameter $\alpha$ increases, $z$ tends to 1, and the solution becomes smooth, eventually approaching the $F_\nabla$ solution with the same value of $a=1$. With smaller regularization parameter values $\alpha = 10^{-4}$ and $\alpha = 10^{-3}$, $z$ becomes more spatially varying and in those areas where $z$ is small, edges are formed in $\gamma$. Setting $\alpha$ too small ($10^{-5}$) leads to a highly unstable solution. 

When compared to the $F_\nabla$ reconstructions, $F_\lambda$ reconstructions with $\alpha = 10^{-3}$ and $\alpha = 10^{-2}$ have sharper edges in locations where $z$ is small. Indeed, these edges are clearly even sharper than those in the $\TV$ regularized solution. Further, in locations where $z$ is closer to one, $F_\lambda$ regularized reconstructions have smooth variations, similar to those in $F_\nabla$ reconstructions. These variations are not present in the $\TV$ reconstructions. Hence, in this case, where the sharp-edged inclusions are placed on a smoothly varying background, $F_\lambda$ provides more suitable regularization than $\TV$ and $F_\nabla$; the relative error of the best $F_\lambda$ reconstruction, $\alpha = 10^{-2}$, is RE$=21.00$ \%. For TV\footnote{For a TV reconstruction with $a=10^{-3}$, which is not shown in \Cref{fig:Case1}, RE=$51.34$ \%.}, these values are $a = 10^{-1}$ and RE=$34.25$ \%. Based on this test, we fix $\alpha = 10^{-2}$ for $F_\lambda$ regularization and $a = 10^{-1}$ for TV.

In \cref{fig:Case1}, the reconstructions with parameters $\alpha = 10^{-3}$ to $\alpha = 10$ display artifacts on the border. These artifacts are most probably caused by the modelling errors due to the differences between the simulation and the inverse mesh. Similar artifacts are also visible in the $F_\nabla$ and $\TV$ regularized solutions, although they are more pronounced in the $F_\lambda$ solutions. These errors could be mitigated, for example, by the so-called approximation error method, but this would require Bayesian formulation of the reconstruction problem, and is out of the scope of this paper \cite{kaipio2006statistical,nissinen2009}.
\begin{table}
    \caption{Description of the numerical and experimental test cases. Here, $\gamma_\mathrm{bg}$ is the background conductivity and $\gamma_\mathrm{ci}$ and $\gamma_\mathrm{ri}$  refer respectively to the non-smooth conductive and resistive inclusions placed on the smooth/constant background. Note that Cases 1-7 use simulated data. Inclusion $\gamma_\mathrm{ci}$ and $\gamma_\mathrm{ri}$ in Case 8 describe cross-section of the objects. }
    \label{table:RE}
    \centering
    \footnotesize%\small
    \begin{tabular}{ r r r r r r r r}
         & Case 1 & Case 2 & Case 3 & Case 4 & Case 5 & Case 6 & Case 7  \\
        \hline
        \hline
        RE ($F_\nabla$) \vline& 21.00 & 27.49 & 19.41 &  37.65  & 16.38 & 15.60 & 17.09\\
        RE ($\TV$) \vline& 34.25 & 38.83 & 19.20 & 41.72 & 17.76 & 28.55 & 16.82
    \end{tabular} 
\end{table}%\todo[inline]{Better explanation for Cases 3-8 in the Table 1?} 
\captionsetup[subfloat]{position=top,labelformat=empty}
\begin{figure*}[!t]
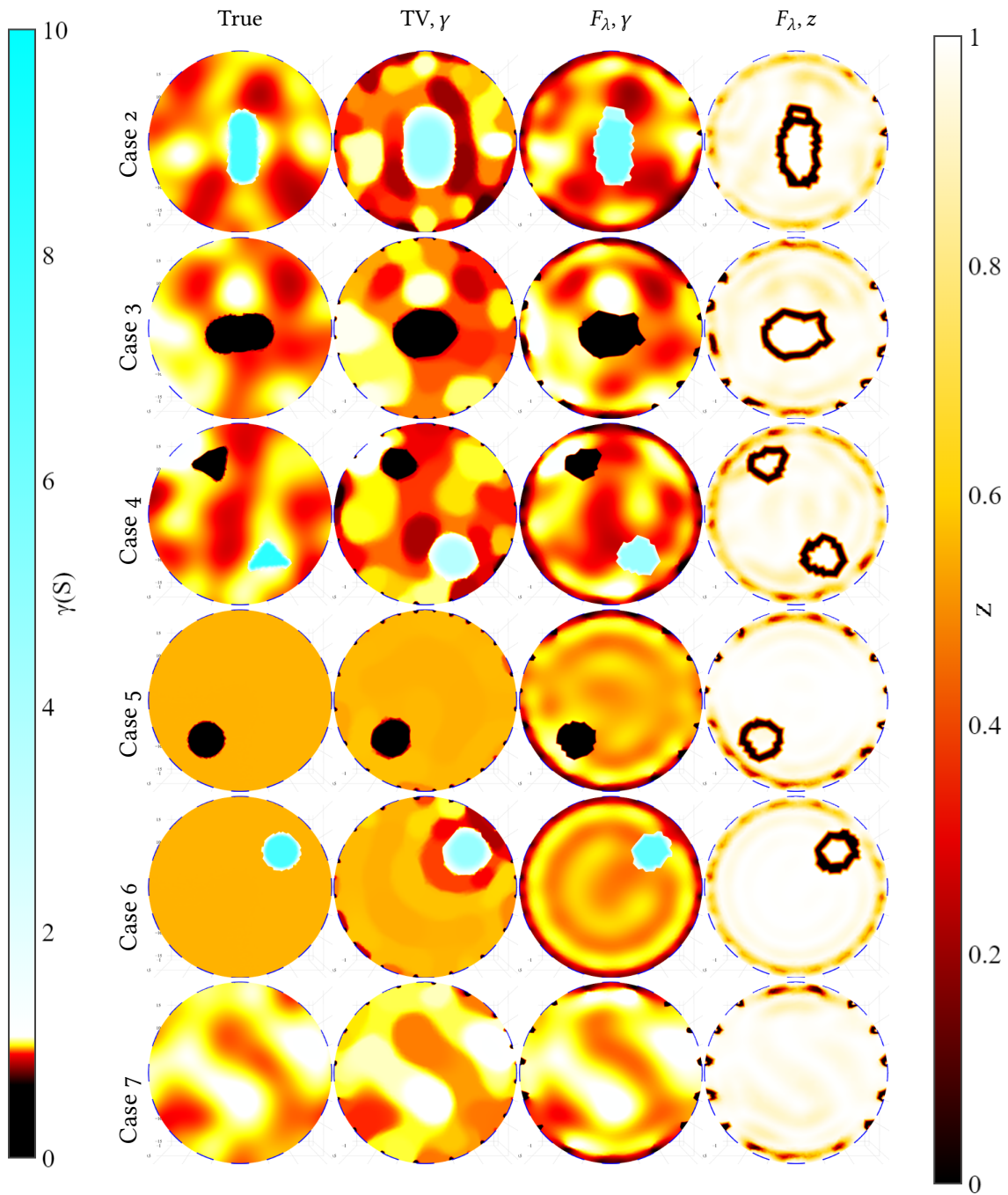
%
\centering%
\subfloatcolorbar{0.0cm}{-8.9cm}{0.082\textwidth}{"Figures/Case0/CBb".png}%    
\begin{minipage}{0.8\textwidth}%
    \centering%
        % \rotatebox{90}{\hspace{-7.9cm}$F_\lambda,z$ \hspace{1.4cm} $F_\lambda,\gamma$ \hspace{1.3cm}  $TV,\gamma$ \hspace{1.2cm} $L^2,\gamma$}
        \rotatebox{90}{\hspace{-16.6cm} Case 7 \hspace{1.7cm} Case 6 \hspace{1.7cm} Case 5 \hspace{1.7cm} Case 4 \hspace{1.7cm} Case 3 \hspace{1.7cm} Case 2}
        \subcstagef{True}{"Figures/Case3/Disk16_Case4_1t".png}{"Figures/Case3/Disk16_Case4_2t".png}{"Figures/Case3/Disk16_Case4_3t".png}{"Figures/Case3/Disk16_Case4_4t".png}{"Figures/Case3/Disk16_Case4_5t".png}{"Figures/Case3/Disk16_Case4_6t".png}% 
        \subcstagef{$\TV,\gamma$}{"Figures/Case3/Disk16_Case4_1_TV".png}{"Figures/Case3/Disk16_Case4_2_TV".png}{"Figures/Case3/Disk16_Case4_3_TV".png}{"Figures/Case3/Disk16_Case4_4_TV".png}{"Figures/Case3/Disk16_Case4_5_TV".png}{"Figures/Case3/Disk16_Case4_6_TV".png}%    
        \subcstagef{$F_\lambda,\gamma$}{"Figures/Case3/Disk16_Case4_1".png}{"Figures/Case3/Disk16_Case4_2".png}{"Figures/Case3/Disk16_Case4_3".png}{"Figures/Case3/Disk16_Case4_4".png}{"Figures/Case3/Disk16_Case4_5".png}{"Figures/Case3/Disk16_Case4_6".png}%     
        \subcstagef{$F_\lambda,z$}{"Figures/Case3/Disk16_Case4_1z".png}{"Figures/Case3/Disk16_Case4_2z".png}{"Figures/Case3/Disk16_Case4_3z".png}{"Figures/Case3/Disk16_Case4_4z".png}{"Figures/Case3/Disk16_Case4_5z".png}{"Figures/Case3/Disk16_Case4_6z".png}%                  
    \end{minipage}% 
    \subfloatcolorbar{0.0cm}{-8.8cm}{0.075\textwidth}{"Figures/Case0/CBzb".png}%  
    \caption{Cases 2-7. Testing $\TV$ and $F_\lambda$ with multiple sets of simulated data. The regularization parameter for $\TV$ is $a=10^{-1}$ and the control parameter and the boundary regularization parameter of $F_\lambda$ are $\lambda = 10^{-3}$ and $\alpha = 10^{-2}$. First column: True conductivity. Second column: $\TV$ reconstruction. Third column: $F_\lambda$ reconstruction of the conductivity. Last column: $F_\lambda$ reconstruction of the control variable.}%
    \label{fig:Case3}%
\end{figure*}%
\captionsetup[subfloat]{position=bottom}

\Cref{fig:Case3} shows the results of Cases 2-7. When the true conductivity has non-homogenous background and contains either a sharped-edged conductive inclusion, resistive inclusion or both (Cases 2-4), the $F_\nabla$ reconstructions have similar features. Further, when the true conductivity has only sharp edges inclusions (Cases 5-6) or only smooth inclusions (Cases 7), the $F_\nabla$ reconstructions also reflect this. 

In the $\TV$ reconstructions, the smooth inclusions are staircased, as expected. \Cref{table:RE} reveals, however, that the staircasing has insignificant impact on RE in Cases 3 and 7, resulting slightly better results with $\TV$ regularization. In these cases, with the chosen parameters, $\TV$ regularization better incorporates the \emph{a priori} information that the conductivity contrast should be small. Nevertheless, the differences in REs of $F_\nabla$ and TV are small. The same applies to Case 5, albeit $F_\nabla$ results in slightly smaller RE than TV. More significantly, in Cases 2, 4, and 6 where the conductivity contrasts are high, $F_\nabla$ regularization clearly outperforms TV. Note that there is some inhomogeneity in the background of Cases 6 and 7, where the true background is constant. Again some errors are present due to the modelling errors.

\begin{remark}The behavior observed in \cref{fig:Case0} implies a mesh refinement scheme through the control variable $\lambda$. The idea is that, if the element sizes within the mesh are fixed, as $\lambda \downto 0$, the cost of setting $z=0$ within an element increases, meaning that in the reconstruction of $z$, the elements where $z=0$ will eventually disappear. Now comparing reconstructions of $z$ with two significantly different $\lambda$ should tell us where the element sizes are too large to set $z=0$, meaning that we need to subdivide these elements. In principle, this can be repeated until numerical precision becomes an issue.\\ We tested this scheme in EIT, but due to relatively low spatial resolution of EIT, this had no noticeable impact on the reconstruction quality. We presume, however, that other imaging modalities would benefit from this scheme.\end{remark}%

\subsection{Experimental studies}\label{ssec:exps}
We also evaluate $F_\lambda$ regularization with experimental data (Case 8). The measurement setup consists of a cylindrical tank that is filled with tap water. The height of the water is $7$ cm and the inner and the outer diameters of the tank are $28.3$ cm and $31.3$ cm, respectively. The tank also has a cylindrical steel rod and a cylindrical plastic rod. The base areas of these cylinders are $7.21$ cm$^2$ and $29.13$ cm$^2$, respectively. Otherwise, this tank matches the geometry used in the numerical case; the 16 electrodes are placed evenly on the border and the length of an electrode is around $1/32$ of the border length (see \Cref{fig:Disk_meas}).

The measurements are taken using an EIT device manufactured by Rocsole Ltd. This device sequentially excites each electrode to a predetermined potential, grounds the others, and measures the currents caused by the potential difference. We used 56kHz excitation frequency. The device samples the currents with $1$ Mhz frequency, and automatically computes the current amplitudes from these samples using FFT. The current amplitudes are read along with the excitation voltages from an ASCII encoded text file.

To evaluate our numerical scheme in practice, we use exactly the same numerical setup as in Cases 2-7 to solve the conductivity and exactly the same parameters for $F_\lambda$. For comparison, we pick TV with $a = 10^{-5}$, as this regularization parameter yielded reconstructions with most accurately sized inclusions. We also show the solution for $a = 10^{-1}$, which yielded the best results in Case 1 and was subsequently used as a comparison in Cases 2-7.

To analyze the reconstructions, we compute the areas of the conductive and resistive inclusions. First we use the half width at half maximum (HWHM) to define the conductive inclusion and then removing the areas associated with the conductive inclusion, we use HWHM again to define the resistive inclusion.
\subsubsection{Results: Experimental data}\label{ssec:resexp}

\begin{table}
    \caption{The true base areas of the conductive and resistive inclusions and the base areas computed from the reconstructions. The conductive inclusion for TV with $a= 10^{-1}$ is defined by $70$ \% of the maximum value since it was more in line with the visual shape. The value computed using HWHM was $608.10$ cm$^2$; almost the area of the whole domain without the resistive inclusion. }
    \label{table:Area}
    \centering
    \footnotesize%\small
    \begin{tabular}{ r c c c c}
         & True & $F_\nabla$ & TV ($a = 10^{-5}$) & TV ($a = 10^{-1}$)   \\
        \hline
        \hline
        Area of the conductive inclusion (cm$^2$) \vline& \AreaRealCon & \AreaMSCon & \AreaTVCon & \AreaTVbCon \\
        Area of the resistive inclusion (cm$^2$) \vline& \AreaRealRes & \AreaMSRes & \AreaTVRes & \AreaTVbRes
    \end{tabular} 
\end{table}%\todo[inline]{Better explanation for Cases 3-8 in the Table 1?} 

\Cref{fig:Disk_meas} shows an $F_{\lambda}$ and two $\TV$ regularized ($a = 10^{-5}$ and $a = 10^{-1}$) solutions computed from the measurement data (Case 8). The background conductivity in these reconstructions is slightly inhomogeneous, although the true background was constant because the tank was filled with properly mixed saline that was at the same temperature as the tank.

The maximum conductivity in the location of the steel bar is about five times larger in the M-S reconstruction than in the TV reconstruction with $a=10^{-5}$ and around seven times larger than in the TV reconstruction with $a=10^{-1}$.

\Cref{table:Area} shows the true base areas of the inclusions as well as the base areas computed from the $F_\nabla$ and the TV reconstructions. The errors in the area estimates for the resistive inclusion are within 14 \% of the real value for all reconstructions: $3.36$ \%, $7.66$ \%, and $13.56$ \% for $F_\nabla$, TV ($a=10^{-5}$), and TV ($a=10^{-1}$), respectively. In $F_\nabla$ reconstruction, the error in the area of the conductivity inclusion is also in a similar range (5.13 \%). For both TV reconstructions, however, area of the conductive inclusion is significantly larger than the real one; in TV reconstruction with $a=10^{-5}$, the inclusion it is around 3 times larger than the real one while with $a=10^{-1}$ it is almost eight times larger. The errors in these base area estimates were 195.15 \% and 653.82 \%, respectively.
\begin{figure}[!tbp]
    \centering
    \begin{minipage}{0.9\textwidth}  
        \centering
        \subfloatrecoib{Photo of the measurement setup}{0.22\textwidth}{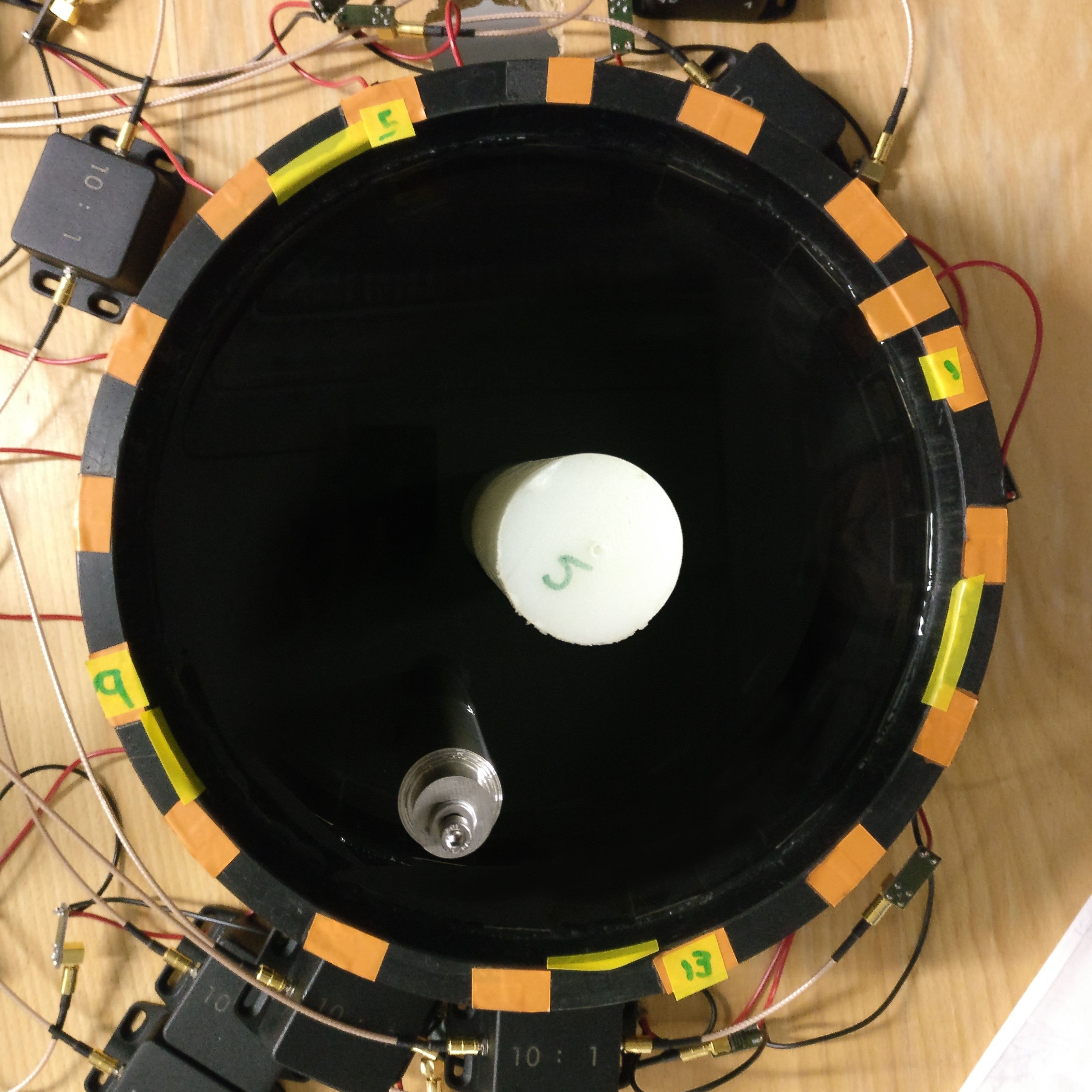}
        \subfloatrecoib{\quad\quad\quad$F_\lambda$ solution\\ \phantom{aaa}($\alpha = 10^{-2}$, $\lambda = 10^{-3}$)}{0.22\textwidth}{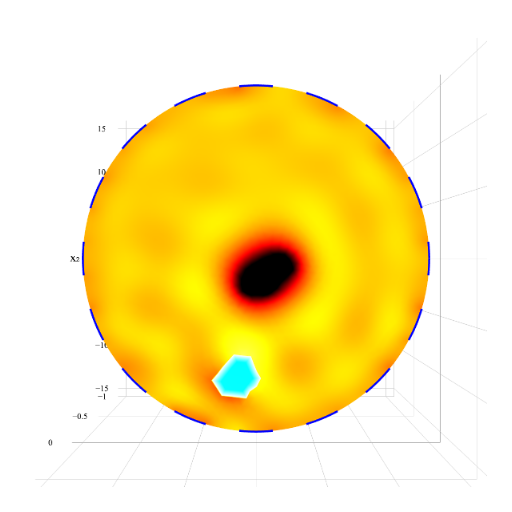}
        \subfloatrecoib{$\TV$ solution ($a = 10^{-5}$)}{0.22\textwidth}{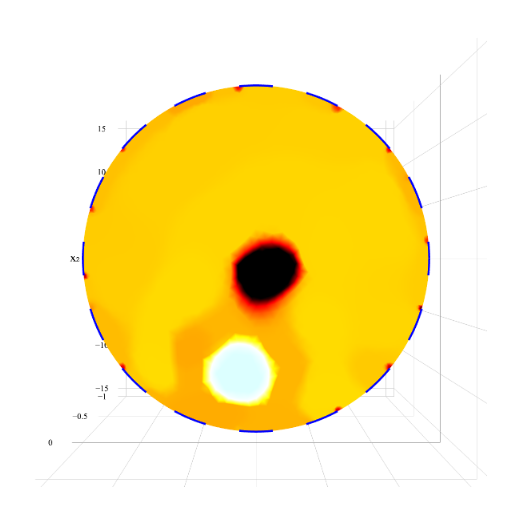}
        \subfloatrecoib{$\TV$ solution ($a = 10^{-1}$)}{0.22\textwidth}{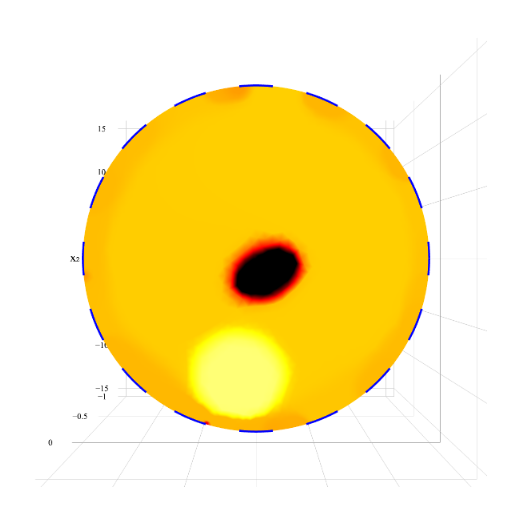}\\     
    \end{minipage}
    \subfloatcolorbar{-0.3cm}{-2.12cm}{1.1cm}{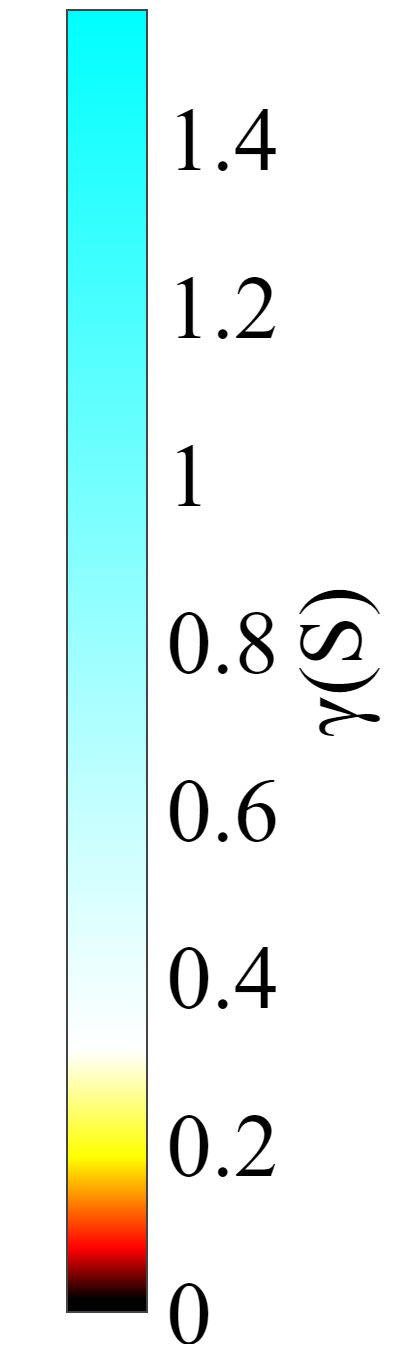}%
    \caption{Case 8. A test with measurement data. Photo of the measurement setup (Left) and reconstructions with $F_\lambda$ and $\TV$ regularization.}
    \label{fig:Disk_meas}
\end{figure}
\subsection{Discussion}
Based on these numerical tests, we conclude that $F_{\lambda}$ is feasible for EIT in cases where the conductivity is expected to contain both smooth and non-smooth inclusions. Especially high contrast inclusions with sharp edges are reconstructed accurately. The $F_{\lambda}$ regularizer seems suitable also when it is not known beforehand which type of these inclusions the target conductivity has; in such cases, even if the true background conductivity is homogenous, some smooth variations maybe still be present in the reconstructions. Note also that in all test cases, the same value of parameter $a$ was used in $F_\lambda$ regularization. Our numerical tests showed that the estimates were fairly robust with respect to the choice of this parameter; for example, setting $a=10$ or $a=10^{-1}$ did not improve the image quality. This was an expected result, because also the $L_2$ regularized solutions tolerate about two orders of magnitude variation in the regularization parameter without significant change in the reconstruction quality (see \cref{fig:Case1}).

\section{Conclusions}

The M-S functional is a popular starting point for image segmentation: finding objects within an image. Also in EIT, the object to be imaged can consist of sharp edged inclusions on a smoothly varying background, or vice versa. Thus far, however, only a few approaches to reconstructing sharp and smooth features in EIT simultaneously have been proposed, or thoroughly studied.

In this paper, we investigated the approximation of the M-S functional in EIT. We showed that, with a small modification, the functional originally proposed by Ambrosio and Tortorelli is applicable to the approximation of the M-S functional in combination with the EIT data fidelity based on the complete electrode model.
Through numerical and experimental studies we showed that the M-S regularization and the approximating functional offer a viable alternative to conventional TV and smooth gradient regularization if the target conductivity is expected to contain both smooth and non-smooth inclusions. Further, the approach was shown to be feasible also when it is not known beforehand whether the target has smooth or non-smooth inclusions.

\appendix 

%%%%%%%%%%%%%%%%%%%%%%%%%%%%%%%%%%%%%%%%%%%%%%%
\section{The proofs of the \texorpdfstring{$\Gamma$}{Gamma}-liminf and \texorpdfstring{$\Gamma$}{Gamma}-limsup inequalities}\label{sec:Gc}
%%%%%%%%%%%%%%%%%%%%%%%%%%%%%%%%%%%%%%%%%%%%%%%

%
In this section, we proved proofs for the $\Gamma$-liminf and $\Gamma$-limsup inequalities for $\bar F_k$. The next lemma shows the $\Gamma$-liminf inequality.

\begin{lemma}\label{lemma:liminf}
    Suppose that $\alpha > 0$, $\epsilon_k \in [0,1]$, and \cref{Ocond} holds. Let $\lbrace (\gamma_k,z_k) \rbrace_{k\in\N} \subset  \mathscr{B}(\Omega) \times \mathscr{B}(\Omega)$ so that $(\gamma_k,z_k)\to (\gamma,z) \in \mathscr{B}(\Omega) \times \mathscr{B}(\Omega)$ in measure. Then the $\Gamma$-liminf inequality \eqref{eq:Gamma1} holds.
\end{lemma}

\begin{proof}
    Suppose that $(\gamma_k,z_k) \to (\gamma,z) \in \mathscr{B}(\Omega) \times \mathscr{B}(\Omega)$ in measure. We may also assume that $\liminf_{k\to\infty}\bar F_k(\gamma_k,z_k) = m < \infty$ as otherwise the inequality holds trivially. Also $m \ge 0$, since $\bar F_k(\gamma,z) \ge 0$ for any $(\gamma,z)$. Since $0 \le m < \infty$, we can find a bounded subsequence $\lbrace \bar F_i(\gamma_i,z_i) \rbrace_{i\in\N}$ such that $\lim_{i\to\infty}\bar F_i(\gamma_i,z_i) = m$ and $\gamma_i \to \gamma$ in measure. The boundedness of $\lbrace \bar F_i(\gamma_i,z_i) \rbrace_{i\in\N}$ implies that $\gamma_i \in [\gamma_m,\gamma_M]$ (almost everywhere) which ensures that $\gamma \in [\gamma_m,\gamma_M]$. Since $\gamma \in [\gamma_m,\gamma_M]$, the $\Gamma$-liminf inequality holds for $F_k$ and $F$ by \cite[Theorem 1.1]{ambrosio1990approximation}. Further, we also have $\bar F_k(\gamma,z) \ge F_k(\gamma,z)$ for any $(\gamma,z) \in \mathscr{B}(\Omega) \times \mathscr{B}(\Omega)$ and $0 \le \epsilon_k \le 1$. Hence
    \begin{equation*}%\label{eq:Gamma11}
        \liminf_{k \to \infty} \bar F_k(\gamma_k,z_k) \ge \liminf_{k \to \infty} F_k(\gamma_k,z_k) \ge F(\gamma,0).
        \qedhere
    \end{equation*}
\end{proof}

Now it remains to prove the existence of a reconstruction sequence for the $\Gamma$-limsup inequality \eqref{eq:Gamma2}. 
Ambrosio and Tortorelli show in \cite[Proposition 5.1]{ambrosio1990approximation} that for any $\gamma \in SBV(\Omega)\cap L^\infty(\Omega)$ and $z=0$ there exist $(y_k, z_k) \in (H^1(\Omega) \times H^1(\Omega) ) \cap ([-\infty,\infty]\times [0,1])$ with $(y_k, z_k) \to (\gamma,0)$ satisfying
\begin{equation} \label{eq:cond51}
    \limsup_{k\to\infty}F_k(\gamma_k,z_k) 
    \le \int_\Omega \abs{\nabla \gamma}^2 \ddd\Lmeas + \limsup_{\rho \downto 0} \frac{\Lmeasof(\lbrace x \in \Omega \mid \dist(x,S_\gamma) 
    < \rho \rbrace)}{2\rho},
\end{equation}
Then, under \cref{Ocond}, they use this estimate to prove $\Gamma$-limsup inequality in \cite[Proposition 5.2-5.3]{ambrosio1990approximation}.
The restrictions $\gamma_k \in [\gamma_m,\gamma_M]$ and $z_k \in [0,1 - \epsilon_k]$ do not hold for the reconstruction sequences of \cite{ambrosio1990approximation} and therefore, to adapt the work of \cite{ambrosio1990approximation} to our modified $\bar F_k$ we need to adapt the reconstruction sequences and proof of \cite[Proposition 5.1]{ambrosio1990approximation}. We do this in the next lemma. We can then use \cite[Proposition 5.3]{ambrosio1990approximation} on this adapted sequence since they do not depend on the explicit form of the reconstruction sequence.

\begin{lemma}\label{lemma:Felimsup}
    Suppose $\Omega \subset \R^{N}$ is an open bounded domain and that $\alpha > 0$. Given any $\gamma \in SBV(\Omega)$ and $\epsilon_k \downto 0$, there exists a sequence $\lbrace(\gamma_k,z_k) \rbrace_{k\in\N} \subset \bar D_{k,N}(\Omega)$ with $(\gamma_k,z_k) \to (\gamma,0)$ in measure and
    \begin{equation}\label{eq:Gamma21}
        \limsup\limits_{k \to +\infty} \bar F_k( \gamma_k, z_k) 
        \le \int_\Omega \abs{\nabla \gamma}^2 \ddd\Lmeas + \limsup_{\rho \downto 0} \frac{\Lmeasof(\lbrace x \in \Omega \mid \dist(x,S_\gamma) < \rho \rbrace)}{2\rho}.
    \end{equation}
\end{lemma}
\begin{proof}
    For simplicity of notation, we assume that $\alpha = 1$.
    Moreover, we may assume that
    \begin{equation}\label{eq:recoas}
        \gamma \in [\gamma_m,\gamma_M],\quad\abs{\nabla \gamma} \in L^2(\Omega),
        \quad\text{and}\quad L 
        \defeq \limsup\limits_{\rho \downto 0} \frac{\Lmeasof(\lbrace x \in \Omega \mid \text{dist}(x,S_\gamma) 
        < \rho \rbrace)}{2 \rho} 
        < \infty,
    \end{equation}
    since otherwise \eqref{eq:Gamma21} holds trivially. If $\gamma \in H^1(\Omega)$, we may simply choose $\gamma_k \equiv \gamma$ and $z_k \equiv 0$ for all $k \in \N$. More generally, for $S_\gamma$ the jumpset of $\gamma$, define the neighborhood
    $$
        (S_\gamma)_{b_k} \defeq \left\lbrace x \in \Omega \mid \text{dist}(x,S_\gamma) < b_k \right\rbrace.
    $$
    For the reconstruction sequence of $z$, we simply restrict $z_k$ chosen by Ambrosio and Tortorelli \cite[Identity 5.5.]{ambrosio1990approximation} above by $1-\epsilon_k$, taking
    \begin{equation}\label{eq:zk}
        z_k = 1-\epsilon_k \text{ on } (S_\gamma)_{b_k},\quad\eta_k\text{ on } \Omega\backslash(S_\gamma)_{a_k+b_k},\quad\text{and}\quad \hat z_k\text{ on }(S_\gamma)_{a_k+b_k}\backslash(S_\gamma)_{b_k},
    \end{equation}
    where $\hat z_k \in [0,1-\epsilon_k]$ will be defined later. We also confirm that $z_k\in H^{1}(\Omega)$. Observe that the definition of $z_k$ extends to $\Omega \cup (S_\gamma)_{a_k+b_k}$. We utilize this property later. During the course of the proof, we will choose the constants $a_k, b_k,\eta_k > 0$ such that $\lim_{k\to \infty}  a_k = 0$,
    \begin{equation}\label{eq:condbk0} %eq:etak
        \lim_{k\to \infty}  k^2b_k = 0,\quad\text{and}\quad \eta_k \defeq \tfrac{1}{k} \sqrt{\int_0^1(1-s^2)^k ds}.
    \end{equation}
    With these choices \cref{eq:recoas} gives 
    \begin{equation}\label{eq:condbk}
        \limsup_{k\to \infty} k^2\Lmeasof(S_\gamma)_{b_k}= 2k^2b_k \Lmeasof(S_\gamma)_{b_k}/(2b_k)=0 \cdot 2L = 0.
    \end{equation} 
    Later we take a specific choice for $a_k$ which reveals the reasoning for the selection of $\eta_k$.

    For the reconstruction sequence of $\gamma$, we take
    \begin{equation}\label{eq:greco}
        \gamma_k \defeq (1-\Psi_k)\gamma + \gamma_m \Psi_k,
    \end{equation}
    for $\Psi_k \in H^1(\Omega)$ satisfying
    $$
        0 \le \Psi_k \le 1,\quad \Psi_k = 1 \text{ on } (S_\gamma)_{b_k/2},\quad\text{and}\quad \Psi_k = 0\text{ on } \Omega\backslash(S_\gamma)_{b_k}.
    $$
    This sequence differs from the one in \cite[page 1025]{ambrosio1990approximation} by the term $\gamma_m \Psi_k$ and by the fact that $\Psi_k$ is not necessarily in $ C^\infty_0(\R^N)$.

    Now, since $ \gamma_k \in [\gamma_m,\gamma_M]$ and $\Lmeasof(S_\gamma)_{b_k/2} \to 0$ due to \eqref{eq:condbk}, we  have $ \gamma_k =  \gamma $ on $\Omega\backslash(S_\gamma)_{b_k}$, $ \gamma_k \to \gamma $ in measure. Further, since $\gamma \in H^1(\Omega\backslash S_\gamma) \cap L^{\infty}(\Omega\backslash S_\gamma)$ and $\Psi_k \in H^1(\Omega) \cap L^\infty(\Omega)$, have $\gamma_k\in H^{1}(\Omega)$ \cite[Theorem 1.49]{KinunenMaI}. By the same theorem, $\gamma \Psi_k$ satisfies Leibniz rule. 

    We follow the Ambrosio's and Tortorelli's proof, however, the bound $z_k\le 1-\epsilon_k$ introduces additional steps. With $(\gamma_k,z_k)$ given by \eqref{eq:greco} and \eqref{eq:zk}, extending $z_k$ to $\Omega \cup (S_\gamma)_{a_k+b_k}$, and estimating $\int_\Omega \frac{1}{4}(\alpha k z)^2 \ddd\Lmeas \le \int_{\Omega \cup (S_\gamma)_{a_k+b_k}} \frac{1}{4}(\alpha k z)^2 \ddd\Lmeas$, etc., we expand
    \begin{equation}\label{eq:expFk}
        \begin{aligned}[t]
            \bar F_k(\gamma_k,z_k) \le& \int_{\Omega\backslash (S_\gamma)_{a_k+b_k}} \abs{\nabla \gamma}^2(1- \eta_k^2)^{2k}  \ddd\Lmeas + \int_{ ((S_\gamma)_{a_k+b_k} \backslash (S_\gamma)_{b_k}) \cap \Omega} \abs{\nabla \gamma}^2(1- z_k^2)^{2k}  \ddd\Lmeas\\
            &+ \int_{ ((S_\gamma)_{b_k}\backslash (S_\gamma)_{b_k/2}) \cap \Omega} \abs{\nabla  \gamma_k}^2(1- (1 - \epsilon_k)^2)^{2k}  \ddd\Lmeas\\
            &+ \tfrac{1}{4} k^2\eta_k^2\Lmeasof(\Omega\backslash(S_\gamma)_{a_k+b_k}) + \tfrac{1}{4}(1-\epsilon_k)^2 k^2\Lmeasof((S_\gamma)_{b_k}) + A_k( z_k),
        \end{aligned}
    \end{equation}
    where the cost of transitioning from $\eta_k$ to $1-\epsilon_k$ in the neighborhood of $S_\gamma$ is presented by the term 
    $$
        A_k(z) \defeq  \int_{ (S_\gamma)_{a_k+b_k}\backslash(S_\gamma)_{b_k}}  \abs{\nabla z}^2(1- z^2)^{2k} + \frac{1}{4}( k z)^2 \ddd\Lmeas.
    $$ 
    
    By closely following the proof of \cite[Proposition 5.1]{ambrosio1990approximation}, since $\eta_k \to 0$ and $(1- \eta_k^2)^{2k} \to 1$ and since \eqref{eq:recoas} implies that $\Omega \cap \bar{S}_\gamma$ is negligible, i.e. it has a zero (Lebesgue) measure, the first term of the RHS converges to $\int_{\Omega} \abs{\nabla \gamma_k}^2  \ddd\Lmeas$. Further, using \eqref{eq:recoas}, \eqref{eq:condbk}, \eqref{eq:condbk0} (i.e. $b_k\to 0$), $a_k\to 0$, we get
    $$
        0 
        \le \limsup\limits_{k\to \infty} \Lmeasof((S_\gamma)_{a_k+b_k}\backslash (S_\gamma)_{b_k})
        \le  \limsup\limits_{k\to \infty} (a_k+b_k)\frac{\Lmeasof((S_\gamma)_{a_k+b_k})}{a_k+b_k} 
        =0\cdot 2L=0.
    $$
    Consequently, $\Lmeasof((S_\gamma)_{a_k+b_k} )\to 0$. The facts $\abs{\nabla \gamma} \in L^2(\Omega)$ and $\Lmeasof((S_\gamma)_{a_k+b_k}\backslash(S_\gamma)_{b_k} )\to 0$ imply that (see \cref{lemma:nunuk} and \cite[Corollary 16.9]{jost2006postmodern})
    $$
        0
        \le\int_{ (S_\gamma)_{a_k+b_k} \backslash (S_\gamma)_{b_k}} \abs{\nabla \gamma}^2(1- z_k^2)^{2k}  \ddd\Lmeas 
        \le \int_{ (S_\gamma)_{a_k+b_k} \backslash (S_\gamma)_{b_k}} \abs{\nabla \gamma}^2 \ddd\Lmeas \to 0.
    $$
    The condition \eqref{eq:condbk} shows that the fifth term of \eqref{eq:expFk} also vanishes. The fourth term may be estimated above by $k^2\eta_k^2\Lmeasof(\Omega)$, which also tends to zero by the definition of $\eta_k$ in \eqref{eq:condbk0}. Thus we only have the third and the last term to estimate.  %eq:etak is used only here

    We next show that $\limsup_{k\to \infty}A_k( z_k) \le L$. For simplicity we write 
    $$
      H_k \defeq (S_\gamma)_{a_k+b_k}\backslash(S_\gamma)_{b_k},\;\tau(x) \defeq \dist(x,S_\gamma),
    \;
    \text{and} 
    \;
        \mathscr{H}(t) \defeq \mathscr{H}^{N-1}(\lbrace y \in \Omega\mid \tau(y) = t\rbrace). 
    $$
    We now take our specific choice of $\hat z_k$ (see \cite[1028]{ambrosio1990approximation}), defined through the single variable function
    \begin{equation}\label{eq:theta}
        \hat z_k(t) \defeq \theta_k(a_k + b_k - t),\quad t\in[b_k,a_k+b_k],
    \end{equation}
    that we parametrize through the distance $t=\tau(y)$ from the jumpset $S_\gamma$. The function $\theta_k$, in turn, we define as the solution of the differential equation
    \begin{equation}\label{eq:gradtheta}
        \nabla \theta_k = \frac{k\theta_k}{2(1-\theta_k^2)^k},\quad\theta_k(0) =\eta_k. 
    \end{equation}
    Observe that $A_k$ only depends on $z_k$ via $\hat z_k$. By using the co-area formula, the fact that $\abs{\nabla \tau} = 1$ a.e. \cite[Theorem 3.2.12 and Lemma 3.2.34]{federer2014geometric}, also \cite[pages 1026-1028]{ambrosio1990approximation}, we can expand $A_k( z_k)$ as
    \begin{equation}\label{eq:Akt}
        \begin{aligned}[t]
            A_k( z_k) 
            &= \int_{b^k}^{a_k + b_k} \int_{\lbrace y \mid \tau(y) = t\rbrace}\left(\left[\abs{\nabla z_k(y)}^2(1-  z_k^2(y))^{2k} + \tfrac{1}{4}k^2 z_k^2(y)\right] d\mathscr{H}^{N-1}(y) \right) dt \\ 
            &=  \int_{b^k}^{a_k + b_k} \left[\abs{\nabla \hat z_k(t)}^2(1- \hat z_k^2(t))^{2k} + \tfrac{1}{4}k^2\hat z_k^2(t)\right] \left( \int_{\lbrace y \mid \tau(y) = t\rbrace}  d\mathscr{H}^{N-1}(y) \right) dt\\
            &= \int_{b^k}^{a_k + b_k} \left[\abs{\nabla \hat z_k(t)}^2(1- \hat z_k^2(t))^{2k} + \tfrac{1}{4}k^2\hat z_k^2(t)\right] \mathscr{H}(t) dt.
        \end{aligned}
    \end{equation}
    In the following, since $\eta_k\to 0$, we may assume that $k$ is large enough so that $\eta_k \le 1-\epsilon_k$. By \eqref{eq:zk} and \eqref{eq:theta}, we have that $\hat z_k(a_k + b_k) = \theta_k(0) = \eta_k$ and $\hat z(b_k) = \theta_k(a_k) = 1-\epsilon_k$. Now dividing \eqref{eq:gradtheta} by its own right-hand side and separating the differential equation gives an implicit formula,%
    $$
    \begin{aligned}[t]
        \int_0^{t} 1 du = \tfrac{2}{k} \int_{\eta_k}^{\theta_k(t)}\tfrac{1}{s} (1-s^2)^k ds.
    \end{aligned}
    $$
    \noindent Further, using $\theta_k(a_k) = 1-\epsilon_k$ gives
    $$
    \begin{aligned}[t]
        a_k &= \int_0^{a_k} 1 du = \tfrac{2}{k} \int_{\eta_k}^{1-\epsilon_k}\tfrac{1}{s} (1-s^2)^k ds \le  \tfrac{2}{k\eta_k}  \int_{\eta_k}^{1-\epsilon_k} (1-s^2)^k ds\\ 
        &\le \tfrac{2}{k\eta_k}  \int_{0}^1 (1-s^2)^k ds = 2  \sqrt{\int_{0}^1 (1-s^2)^k ds},
    \end{aligned}
    $$
    where on the last equality we used the definition of $\eta_k$. By change of variables and Hölder's inequality it is easy to see that $a_k \to 0$. Also note that as $\hat z_k$ is decreasing, $\hat z_k \in [\eta_k,1-\epsilon_k]$. 

    Using \eqref{eq:gradtheta} gives
    $
        (1-\hat z_k^2)^k \nabla \hat z_k = - k \hat z/2,
    $
    and plugging this into \eqref{eq:Akt} gives
    \begin{equation}\label{eq:Akzk}
        \begin{aligned}[t]
            A_k( z_k) &= \int_{b^k}^{a_k + b_k} \left[\abs{\nabla \hat z_k(t)}^2(1- \hat z_k^2(t))^{2k} + \tfrac{1}{4}k^2\hat z_k^2(t)\right] \mathscr{H}(t) dt 
            = \tfrac{k^2}{2} \int_{b^k}^{a_k + b_k} \hat z_k^2(t) \mathscr{H}(t) dt.
        \end{aligned}
    \end{equation}
    Similarly to \cite{ambrosio1990approximation}, in the following, we will denote terms that vanish as $k \to \infty$ by $o(k)$ and define
    \begin{equation}\label{eq:deffA}
        \A(t) \defeq \Lmeasof(\lbrace x\in \Omega \mid \dist(x,S_\gamma) < t\rbrace).
    \end{equation}
    By \cite[Identity 2.6]{ambrosio1990approximation}, $\A \in W_\text{loc}^{1,1}((0,\infty))$ and $\nabla \A = \mathscr H$ almost everywhere (i.e. $\A(s) = \int_0^s \mathscr{H}(t)dt$), so that integration by parts gives
    $$
        A_k( z_k) = \tfrac{k^2}{2}\left( \hat z_k^2(t)\A(t)\Big|_{b_k}^{a_k+b_k} -2 \int_{b_k}^{a_k+b_k} \hat z_k(t) (\nabla\hat  z_k(t))\A(t) dt \right)
    $$
    Recalling that $\hat z_k(b_k) = 1-\epsilon_k$ and $\hat z_k(a_k+b_k)= \eta_k$,
    $$
    \begin{aligned}[t]
        A_k( z_k) &=  \tfrac{k^2}{2}\left( \eta_k^2\A(a_k+b_k) - (1-\epsilon_k)^2\A(b_k) -2 \int_{b_k}^{a_k+b_k} \hat z_k (\nabla\hat  z_k(t))\A(t) dt \right) \\
        &\le -k^2\int_{b_k}^{a_k+b_k} \hat z_k(t) (\nabla\hat  z_k(t))\A(t) dt + o(k),
    \end{aligned}
    $$
    where the term $\eta_k^2\A(a_k+b_k)$ is $o(k)$ due to the definition of $\eta_k$ and $- (1-\epsilon_k)^2\A(b_k) \le 0$ (and $o(k)$).

    Further, \eqref{eq:recoas} ensures the existence of $\omega_k \to 0$ such that $\A(t) \le 2t(L+\omega_k)$ for all $t\in[0, a_k + b_k]$.
    Indeed, take $\omega_k \defeq \beta_k -L$ for $\beta_k \defeq \sup_{s \in [0,a_k+b_k]} \A(s)/(2s)$. Since $t\in [0,a_k+b_k]$, $\A(t)/(2t) \le \beta_k$, moreover, since $a_k+b_k \to 0$, by \eqref{eq:recoas} $\lim_{k\to\infty} \beta_k=\limsup_{k\to\infty} \A(a_k+b_k)/(2(a_k+b_k))=L$. Thus $\omega_k \to 0$. Now $\A(t)/(2t) - \beta_k \le 0$ if and only if $A(t) \le 2t(L+\omega_k)$. Since $\A(t) \le 2t(L+\omega_k)$ for all $t\in[0, a_k + b_k]$ and since $\nabla\hat z_k(t) \le 0$ for a.e.~$t\in[b_k,a_k+b_k]$,
    $$
        A_k( z_k) \le -2(L+\omega_k)k^2\int_{b_k}^{a_k+b_k} \hat z_k(t) (\nabla\hat  z_k(t))t dt + o(k),
    $$
    and integration by parts gives% \todo{ $d/dt (\hat z^2) = 2 (\nabla \hat z) \hat z$, thus $2$ disappears.}
    $$
    \begin{aligned}[t]
        A_k( z_k) &\le (L+\omega_k)k^2 \left(-\eta_k^2(a_k+b_k) + (1-\epsilon_k)^2b_k + \int_{b_k}^{a_k+b_k} \hat z_k^2(t) dt \right) + o(k) 
        \\
        &= (L+\omega_k)k^2\int_{b_k}^{a_k+b_k} \hat z_k^2(t) dt + o(k),
    \end{aligned}
    $$
    where the first term is clearly below zero and the second term is $o(k)$ due to the definition of $b_k$. Finally plugging in $k\hat z_k = -2(\nabla \hat z_k)(1-\hat z_k^2)^k$ from \eqref{eq:gradtheta} gives
    $$
        \begin{aligned}[t]
            A_k( z_k) \le & (L+\omega_k)k\int_{b_k}^{a_k+b_k}  2 \hat z_k(t)(-\nabla \hat z_k(t))(1-\hat z_k^2(t))^k dt + o(k) \\
            &=  \frac{(L+\omega_k)k}{k+1} (1-\hat z_k^2(t))^{k+1}\Big|_{b_k}^{a_k+b_k} + o(k)
            \\
            &= \frac{(L+\omega_k)k}{k+1}\left(  (1- \eta_k^2)^{k+1} -(\epsilon_k(2-\epsilon_k))^{k+1}   \right) + o(k).
        \end{aligned}
    $$
    Since $\omega_k\to 0$, $k/(k+1)\to 1$, $(1-\eta_k^2)^{k+1} \to 1$, and $(\epsilon_k(2-\epsilon_k))^{k+1} \to 0$, we have confirmed that $\limsup_{k\to \infty} A( z_k) \le L$. 
    Since $\norm{z_k}_2$ and $\norm{\nabla z_k}_2$ are bounded, $ z_k \in H^1(\Omega)$. Further, by the same reasoning as we used in \Cref{thm:FGaexists}, also $\phi \circ z_k \in H^1(\Omega)$.

    Next we show that also the third term in \eqref{eq:expFk} vanishes. Denote $K_k \defeq (S_\gamma)_{b_k}\backslash (S_\gamma)_{b_k/2}$. By construction $z_k=1-\epsilon_k$ on $(S_\gamma)_{b_k}$. Thus using \eqref{eq:greco}, Leibniz rule, and Young's inequality gives
    \begin{equation}\label{eq:gammakapp}
        \begin{aligned}[t]
            \int_{ K_k \cap \Omega} \abs{\nabla  \gamma_k}^2(1-z_k^2)^{2k}  \ddd\Lmeas 
            & = \int_{ K_k\cap \Omega} \abs{\nabla  \gamma_k}^2(2\epsilon_k -\epsilon_k^2 )^{2k}  \ddd\Lmeas \\ 
            & \le \int_{ K_k\cap \Omega} 2(\abs{\nabla \gamma}^2 (1-\Psi_k)^2 + (\gamma_m-\gamma)^2\abs{\nabla  \Psi_k}^2 )(2\epsilon_k -\epsilon_k^2 )^{2k}   \ddd\Lmeas,\\
            & \le \int_{ K_k\cap \Omega} 2(\abs{\nabla \gamma}^2 + (\gamma_M-\gamma_m)^2\abs{\nabla  \Psi_k}^2 )(2\epsilon_k -\epsilon_k^2 )^{2k}   \ddd\Lmeas.
        \end{aligned}
    \end{equation}
    Since by \eqref{eq:recoas} $\abs{\nabla\gamma} \in L^2(K_k\cap \Omega)$, and since $\Lmeasof((S_\gamma)_{b_k})\to 0$, by \cite[Corollary 16.9]{jost2006postmodern}, $\int_{K_k\cap \Omega} 2\abs{\nabla  \gamma}^2(2\epsilon_k -\epsilon_k^2 )^{2k}  \ddd\Lmeas \to 0$, regardless of $\epsilon_k \in [0,1)$. To simplify the notations we take $c \defeq 2(\gamma_M-\gamma_m)^2$ and $\ee \defeq 2\epsilon_k -\epsilon_k^2$.

    Similarly to the construction of $\hat z_k$, we choose $\Psi_k$ to be only a function of $\tau$ on $K_k$, $\Psi_k(y) = \hat \Psi_k(\tau(y))$ for some $\hat \Psi_k$. Again using the co-area formula and $\abs{\nabla \tau(x)} = 1$ a.e., we again write
    \begin{equation}\label{eq:PsiH}
        \begin{aligned}[t]
            \int_{ K_k\cap \Omega} \abs{\nabla \Psi_k(x)}^2 \ee^{2k}  \ddd\Lmeas  \le \int_{ K_k} \abs{\nabla \Psi_k(x)}^2 \ee^{2k}  \ddd\Lmeas =   \ee^{2k}\int_{b^k/2}^{b_k} \abs{\nabla \hat \Psi_k(t)}^2 \mathscr{H}(t) dt
    \end{aligned}
    \end{equation}
    where $\hat \Psi_k$ is a locally Lipschitz function only depending on $t$. Before writing explicit formula for $\Psi_k$, let us examine the sequence $b_k$. Note that $b_k$ is an arbitrary sequence that satisfies \eqref{eq:condbk0}. It is easy to see that this is satisfied by $b_k = 1/k^{2+\delta}$ with $\delta > 0$, as soon as $k > k_0$ with large enough $k_0$ so that $(S_\gamma)_{b_k} \subset \Omega$. We will fix $b_k \defeq 1/k^{2+\delta}$ but we let $\delta > 0$ to be arbitrary. We also assume that $k \ge k_0$. 

    Now recall that $\Psi_k = 1$ on $ (S_\gamma)_{b_k/2}$, $\Psi_k = 0$ on $\Omega\backslash(S_\gamma)_{b_k}$, $ \Psi_k \in H^1(\Omega)$, and further, $\abs{\nabla \Psi_k}^2$ has to be locally Lipschitz due to the usage of the co-area formula. These conditions imply that $\nabla \hat \Psi_k \vert_{\partial (S_\gamma)_{b_k/2}} = \nabla \hat \Psi_k \vert_{\partial (S_\gamma)_{b_k}} = 0$ a.e. and thus a simple choice that satisfies them is a piecewise polynomial function
    \begin{equation}\label{eq:hPsi}
        \hat \Psi_k(t) \defeq 
        \begin{cases} 
         1,& t \in [0,b_k/2)\\ 
        (4(b_k - t)^2(4t- b_k ))/b_k^3,& t\in [b_k/2,b_k], \\ 
         0,& \text{otherwise.}
        \end{cases}
    \end{equation}
    Again, since $\abs{\nabla \tau(x)} = 1$ a.e., $\abs{\nabla \hat \Psi_k(t)} = (24(-b_k^2+3b_kt-2t^2))/b_k^3$
     a.e. when $t\in [b_k/2,b_k]$. This polynomial has maximum at $\tfrac{3}{4}b_k$, thus $\norm{\abs{\nabla \hat \Psi_k}^2}_{L^\infty([b_k/2,b_k])} = 9/b_k^2$. 
     Then
    \begin{equation}\label{eq:Phisup}
        \begin{aligned}[t]
            c \ee^{2k}\int_{b^k/2}^{b_k} \abs{\nabla \hat \Psi_k(t)}^2 \mathscr{H}(t) dt \le \frac{9c \ee^{2k}}{b_k^2}\int_{b^k/2}^{b_k} \mathscr{H}(t) dt \le \frac{9c \ee^{2k}}{b_k^2}\int_{0}^{b_k} \mathscr{H}(t) dt = \frac{9c \ee^{2k}}{b_k^2} \A(b_k). 
        \end{aligned}
    \end{equation}
    since $\nabla \mathcal A = \mathscr H$ a.e., and since by \eqref{eq:recoas} $\A(0)=0$. Also by \eqref{eq:recoas}, $\limsup_{k\to\infty}\A(b_k)/b_k = L$, meaning that need to only consider the term $\ee^{2k}/b_k$. 
    Now let $n \in \N$ so that $n \ge \delta$. Then $1/b^k =k^{2+\delta} \le k^{2+n}$. Since $\epsilon_k < 1$, also $\ee < 1$, an application of L'H\^{o}spital rule shows $\ee^{2k}/k^{-2-n} \to 0$ as $k\to\infty$, i.e.,
    $$
        \limsup_{k\to\infty}\ee^{2k}c\int_{ K_k} \abs{\nabla \Psi_k}^2  \ddd\Lmeas \le \limsup_{k\to\infty} 9c{k^{2+n}\ee^{2k}}\A(b_k)/{b_k}  = 9c (0 \cdot L)=0 .
    $$
    Since all of the additional terms in $\bar F_k( \gamma_k,  z_k)$ vanish, we have shown that
    $$
        \limsup_{k\to\infty} \bar F_k( \gamma_k,  z_k) = \limsup_{k\to\infty} F_k(\gamma_k, z_k) \le F(\gamma,0),
    $$
    as $\epsilon_k \downto 0$. This finishes the proof.
\end{proof}
\noindent \Cref{lemma:Felimsup} yields a reconstruction sequence for any $\gamma \in \SBV(\Omega)$ for which the Minkowski upper limit
$$
    L(S_{\gamma}) \defeq \limsup_{\rho \downto 0} \frac{\Lmeasof(\lbrace x \in \Omega \mid \dist(x,S_\gamma) < \rho \rbrace)}{2\rho}
$$ 
satisfies $L(S_\gamma)\le \mathscr{H}^{N-1}(S_\gamma)$ or $F(\gamma) = \infty$.
Next we extend this to $\SBV(\Omega)$ with arguments similar to \cite[Propositions 5.2-3]{ambrosio1990approximation}.

\begin{lemma}\label{corol:GSBV}
    Suppose that $\alpha > 0$, $\Omega \subset \R^N$ is open and bounded, and that \cref{Ocond} holds. Then given any $\gamma \in SBV(\Omega)$ there exists a reconstruction sequence $\lbrace(\gamma_k,z_k) \rbrace_{k \in \N} \subset \bar D_k(\Omega)$ with $\epsilon_k \downto 0$, such that \eqref{eq:Gamma2} holds.
\end{lemma}
\begin{proof}
    We can focus on the case $\gamma \in \SBV(\Omega) \cap [\gamma_m,\gamma_M]$ as otherwise $F(\gamma) = \infty$. 
    We define the class of functions
    \begin{equation*}
        \begin{aligned}
            \mathscr F(\Omega)\defeq
            &\left\lbrace \gamma \in \GSBV(\Omega) \mid 
             F(\gamma,0) = \infty
            \text{ or }
            \text{there exists } \lbrace v_\ell \rbrace_{\ell\in\N} \subset SBV(\Omega)\cap L^\infty(\Omega),\right.\\%\phantom{\frac{1}{1}}
            & \lim_{\ell\to\infty}L(S_{v_\ell}) - \mathscr H^{N-1}(S_{v_\ell}) = 0 \text{, } \tilde F(v_\ell,0) \to \tilde F(\gamma,0),\\
            &\left. \text{and } v_\ell \to \gamma\text{ in measure.}
            \right\rbrace
        \end{aligned}
    \end{equation*}
    where $\tilde F$ is the Mumford-Shah functional without the constraints of $\gamma$. Since \cref{Ocond} holds, \cite[Proposition 5.3]{ambrosio1990approximation} gives $\mathscr F(\Omega) \cap [\gamma_m,\gamma_M] = \SBV(\Omega) \cap [\gamma_m,\gamma_M]$. Indeed, the sequences $\lbrace v_\ell \rbrace_{\ell\in N} \subset \SBV(\Omega) \cap L^\infty(\Omega)$ with $v_\ell \to \gamma \in \mathscr{F}(\Omega)$ that satisfy the conditions of $\mathscr{F}(\Omega)$ are solutions to
    \begin{equation}\label{eq:seqmin}
        \min_{v \in SBV(\Omega')} \int_{\Omega'} \abs{\nabla v}^2 dx + \mathscr{H}^{N-1}(S_v) + \ell\int_{\Omega'}\abs{v-\gamma}^2dx,
    \end{equation}
    where $\Omega' = \Omega \cup U$ for $U$ the neighbourhood from \cref{Ocond} and $\gamma$ is extended to $\Omega'$ be reflecting with $\phi$.
    Since $\gamma \in [\gamma_m,\gamma_M]$, also $v_\ell \in [\gamma_m,\gamma_M]$, meaning that $\tilde F(\gamma) = F(\gamma)$ and $\tilde F(v_\ell) = F(v_\ell)$.

    Thus, for all $\gamma \in \SBV(\Omega) \cap [\gamma_m,\gamma_M]$ we can find a sequence $\lbrace v_\ell\rbrace_{\ell\in \N} \subset \SBV(\Omega) \cap [\gamma_m,\gamma_M]$ such that
    \begin{equation}
        \label{eq:diagonal-base}
        v_\ell \to \gamma \text{ in measure},
        \quad\lim_{\ell\to\infty}L(S_{v_\ell}) - \mathscr H^{N-1}(S_{v_\ell}) = 0,
        \quad\text{and}\quad F(v_\ell,0) \to F(\gamma,0).
    \end{equation}
    For each $\ell \in \N$, \cref{lemma:Felimsup} now gives a sequence $\lbrace(\tilde  v_{\ell,j},\tilde z_{\ell,j}) \rbrace_{j\in\N} \subset \bar D_{k,N}(\Omega)$ such that
    \begin{equation}
        \label{eq:diagonal-start}
        (\tilde v_{\ell,j},\tilde z_{\ell,j}) \to (v_\ell,0)
        \quad\text{and}\quad
        \limsup_{j \to +\infty} \bar F_j(\tilde v_{\ell,j},\tilde z_{\ell,j})
        \le \int_\Omega \abs{\nabla v_\ell}^2 \ddd\Lmeas + L(S_{v_\ell}).
    \end{equation}
    Since convergence in measure is metrizable on bounded domains $\Omega$, a diagonal argument now establishes a diagonal sequence $\{(\gamma_k, z_k)\}_{k \in \N}$, obtained for some $\{(\ell_k, j_k)\}_{k \in \N}$ as $(\gamma_k, z_k) \defeq (\tilde v_{\ell_k,j_k},\tilde z_{\ell_k,j_k})$, satisfying the $\Gamma$-limsup inequality. Indeed, by the metrizability,
    $
        (\tilde v_{\ell_k,j_k},\tilde z_{\ell_k,j_k}) \to (\gamma,0)
    $
    while using \eqref{eq:diagonal-start} and the definition of $\mathscr F(\Omega)$ and finishing with \eqref{eq:diagonal-base} yields
    \[
        \begin{aligned}
        \limsup_{k \to +\infty} \bar F_j(\tilde v_{\ell_k,j_k},\tilde z_{\ell_k,j_k})
        &
        \le \limsup_{\ell \to +\infty}\left(  \int_\Omega \abs{\nabla v_\ell}^2 \ddd\Lmeas + L(S_{v_\ell})\right)
        \\
        &
        = \limsup_{\ell \to +\infty}\left(  \int_\Omega \abs{\nabla v_\ell}^2 \ddd\Lmeas + {\mathscr H}^{N-1}(S_{v_{\ell}})\right).
        \\
        &
        =  \limsup_{\ell \to +\infty} F(v_\ell,0)
        \\
        &
        = F(\gamma,0).
        \mbox{\qedhere}
        \end{aligned}
    \]
\end{proof}

\bibliographystyle{jnsao}
\bibliography{abbrevs,mumfordshaheit}

%%%%%%%%%%%%%%%%%%%%%%%%%%%%%%%%%%%%%%%%%%%%%%%
\end{document}